\documentclass[letter,11pt]{amsart}
\usepackage{amsmath,amsthm,amssymb,amsfonts}
\usepackage[pdftex]{graphicx}
\usepackage{tikz}
\usepackage[margin=1in,top=1in]{geometry}
\usepackage[latin1]{inputenc}
\usepackage{enumitem}
\usepackage{hyperref}

\newcommand{\R}{\mathbb{R}}
\newcommand{\vep}{\varepsilon}
\newcommand{\osc}{\operatorname{osc}}

\newcommand{\tr}[1]{\operatorname{Tr}(#1)}
\newcommand{\diam}[1]{\operatorname{diam}(#1)}
\newcommand{\chara}[1]{\chi_{#1}}
\newcommand{\id}{I}

\newcommand{\ellmat}{\mathcal{A}_{\lambda, \Lambda}}
\newcommand{\ellvec}{\mathcal{O}_{\lambda,\Lambda,\mu}}
\newcommand{\gcal}{\mathcal{Q}}
\newcommand{\mcal}{\mathcal{M}}
\newcommand{\tcal}{\mathcal{T}}
\newcommand{\scal}{\mathcal{S}}
\newcommand{\mao}{\alpha_{\star}}

\newtheorem{thm}{Theorem}[section]

\newtheorem{cor}[thm]{Corollary}
\newtheorem{lem}[thm]{Lemma}
\theoremstyle{definition}
\newtheorem{defn}[thm]{Definition}
\newtheorem{rem}[thm]{Remark}

\allowdisplaybreaks

\numberwithin{equation}{section}

\author[I. U. Erneta]{I\~{n}igo U. Erneta}
    \address{I\~{n}igo U. Erneta. Department of Mathematics\\
    Rutgers University\\
110 Frelinghuysen Rd., Piscataway, NJ 08854, USA}
    \email{iu40@math.rutgers.edu}

 \title{Regularity theory for fully nonlinear oblique transmission problems}

\begin{document}

\begin{abstract}
We develop a regularity theory for fully nonlinear oblique transmission problems.
Our main result establishes the optimal regularity of viscosity solutions for flat interfaces.
For constant-coefficient equations, we prove that viscosity solutions are piecewise $C^{1,\alpha}$ up to the interface.
The same regularity continues to hold for variable-coefficient problems, with sufficiently regular coefficients, in a more restrictive class of viscosity solutions.
This is the first regularity result for variable-coefficient transmission conditions.
\end{abstract}

\maketitle

\tableofcontents

\section{Introduction}

Transmission problems model systems where the configuration space is composed of multiple pieces or \emph{phases}, each characterized by different properties.
These problems arise naturally in applications involving some heterogeneity, such as composite materials in Physics, regime-switching models in Economics, or biological tissue in Biology.
From an analytical perspective, the abrupt changes in phase behavior across a separating interface are modeled by partial differential equations (PDE) with discontinuous coefficients.
This situation requires prescribing certain coupling criteria on the interface, known in the literature as transmission conditions.

\medskip

Recently in~\cite{E} we introduced a novel class of \emph{fully nonlinear oblique} transmission problems.
The defining characteristic of this class is that the transmission condition depends on the full gradient of the solution from each side of the interface.
Our model can be understood as a nonlinear coupling of oblique derivatives across the interface, whereas all previous works only considered linear Neumann conditions involving normal derivatives.

\medskip

In~\cite{E}, we developed the foundational viscosity solution framework, proving the well-posedness of constant-coefficient problems with flat interfaces.
Building upon those existence and uniqueness results, the present paper offers a comprehensive regularity theory for both constant- and variable-coefficient oblique transmission problems. 
Our central contribution establishes the piecewise $C^{1,\alpha}$ regularity of viscosity solutions, which is the best possible regularity without further assumptions.
We wish to point out that ours is the first regularity result for variable-coefficient transmission conditions, oblique or otherwise.

\medskip

\subsection{Main results}

Let $B_R = \{x \in \R^n \colon |x| < R\} \subset \R^{n}$ be the open ball of radius $R > 0$ centered at the origin.
We think of $B_R$ as being divided into two phases, the upper and lower half-balls
\[
B_{R}^{+} = B_R \cap \{x_n > 0\}\qquad \text{ and } \qquad B_{R}^{-} = B_R \cap \{x_n < 0\}.
\]
These phases are separated by the flat interface
\[
T_R = B_R \cap \{x_n = 0\}.
\]
For functions $u \colon B_R \to \R$, we use the superscript notation $u^{\pm} := u \chara{B_R^{\pm} \cup T_R}$ to denote the restrictions to each half-ball, up to the interface.

\medskip

We consider the transmission problem
\[
\begin{cases}
F^{+}(D^2 u, x ) = f^{+}(x) & \text{ in } B_R^{+},\\
F^{-}(D^2 u, x ) = f^{-}(x) & \text{ in } B_R^{-},\\
G(\nabla u^{+}, \theta \nabla u^{-}, x) =g(x) & \text{ on } T_R,\\
u^{+} = u^{-} & \text{ on } T_R,\\
\end{cases}
\]
where $F^{\pm}$ and $G$ are fully nonlinear operators (see Section~\ref{sec:visc}), $f^{\pm}$ are functions in $B_R^{\pm}$, $g$ is a function on $T_R$, and $0 \leq \theta \leq 1$ is a constant.

\medskip

Here, the second order equations in each phase are known as \emph{bulk equations}, while the first order identity on the interface is the \emph{transmission condition}.
The \emph{interpolation parameter} $\theta$
interpolates between a ``pure'' transmission problem ($\theta = 1$) and a degenerate one ($\theta = 0$) that decouples into two boundary-value problems on each phase (for the details, see Proposition 2.17 in~\cite{E}).
Finally, the last equation serves to indicate that $u$ must be continuous across the interface and will be assumed implicitly throughout the work.

\medskip

Our transmission operator $G$ can be interpreted as a nonlinear coupling of oblique derivatives, generalizing the linear condition
\[
G(\nabla u^{+}, \nabla u^{-}, x) = \gamma^{+}(x) \cdot \nabla u^{+} - \gamma^{-}(x) \cdot \nabla u^{-}.
\]
Such operators are characterized by three ellipticity constants $0 < \lambda \leq \Lambda$ and $\mu \geq 0$ quantifying the size of the coefficients.
Namely, writing $\gamma^{\pm} = (\gamma'^{\pm}, \gamma^{\pm}_{n}) \in \R^{n-1} \times \R$, we impose that:
\begin{itemize}
\item 
the normal components are controlled above and below by 
\[
\lambda \leq \gamma_{n}^{\pm}(x) \leq \Lambda \qquad \text{ for } x \in T_R;
\] 
\item 
the tangential components are bounded by
\[
|\gamma'^{\pm}(x)| \leq \mu \qquad \text{ for } x \in T_R.
\]
\end{itemize}
We give a precise definition of our nonlinear operators and the notion of uniform ellipticity in Section~\ref{sec:visc} below.

\medskip

Throughout the text, a constant is called \emph{universal} if it depends only on $n$, $\lambda$, $\Lambda$, and $\mu$.
We note in particular that a universal constant does not depend on the interpolation parameter $\theta$.
Somewhat surprisingly, our estimates in the theorems below are all independent of this parameter.

\medskip

Our first result provides a H\"{o}lder estimate for viscosity solutions to variable-coefficient problems.
For the definition of viscosity solution, see Definition~\ref{def:visc} below:

\begin{thm}
[H\"{o}lder estimate for variable coefficients]
\label{thm:holder:intro}
Let $\{F^{\pm}, G\}$ be continuous operators in $B_1^{\pm} \cup T_1$ and $T_1$, respectively, and with ellipticity constants $0 < \lambda \leq \Lambda$ and $\mu \geq 0$.
Let $\{f^{\pm}, g\}$ be continuous functions in $B_1^{\pm} \cup T_1$ and $T_1$, respectively.
Let $0 \leq \theta \leq 1$ be a constant.
Let $u \in C(B_1)$ be a viscosity solution to
\[
\begin{cases}
F^{\pm}(D^2u, x) = f^{\pm}(x) & \text{ in } B_1^{\pm},\\
G(\nabla u^{+}, \theta \nabla u^{-}, x) =g(x) & \text{ on } T_1.
\end{cases}
\]
Assume that $F^{\pm}(0,x) = 0$ for $x \in B_1^{\pm}$, and $G(0,0,x) = 0$ for $x \in T_1$.

There are universal constants $0 < \alpha < 1$ and $C > 0$ such that $u$ is $C^{\alpha}$ in $B_1$ and 
\[
\|u\|_{C^{\alpha}(\overline{B}_{3/4})} 
\leq C \left( \|u\|_{L^{\infty}(B_1)}
+\|f^{+}\|_{L^n(B_1^{+})} + \|f^{-} \|_{L^n(B_1^{-})} +\|g \|_{L^{\infty}(T_1)}
\right).
\]
\end{thm}
\begin{rem}
The theorem continues to hold without assumptions $F^{\pm}(0,x) = 0$ and $G(0,0,x) = 0$,
replacing the functions $f^{\pm}$ and $g$ in the estimate by $f^{\pm} - F^{\pm}(0, \cdot)$ and $g - G(0,0,\cdot)$, respectively.
\end{rem}

Theorem~\ref{thm:holder:intro} is the fundamental ingredient for the whole regularity theory, as explained later.
In fact, we will prove a more general result, Theorem~\ref{thm:holder}, where the $C^{\alpha}$ estimate is shown to hold for functions in a universal ``solution class'' (Section~\ref{sec:visc}).
Our H\"{o}lder estimate is a consequence of a new boundary Harnack inequality obtained via appropriate piecewise quadratic barriers (see Section~\ref{sec:holder}). 

\medskip

Our second result gives the optimal regularity of constant-coefficient homogeneous problems:

\begin{thm}
[Piecewise $C^{1,\overline{\alpha}}$ regularity for constant coefficients]
\label{thm:piecewise}
Let $\{F^{\pm}, G\}$ be constant-coefficient operators with ellipticity constants $0 < \lambda \leq \Lambda$ and $\mu \geq 0$.
Let $0 \leq \theta \leq 1$ be a constant.
Let $u \in C(B_1)$ be a viscosity solution to 
\[
\begin{cases}
F^{\pm}(D^2u) = 0 & \text{ in } B_1^{\pm},\\
G(\nabla u^{+}, \theta \nabla u^{-}) =0 & \text{ on } T_1.
\end{cases}
\]
There are universal constants $0 < \overline{\alpha} < 1$ and $C > 0$ such that
$u^{\pm}$ are $C^{1,\overline{\alpha}}$ in $B_1^{\pm} \cup T_1$ and
\[
\|u^{\pm}\|_{C^{1,\overline{\alpha}}(\overline{B}^{\pm}_{1/2})} 
\leq C \left( \|u\|_{L^{\infty}(B_1)} + |F^{+}(0)| + |F^{-}(0)| + |G(0,0)|
\right).
\]
Moreover, $u$ satisfies the transmission condition in the classical sense.
\end{thm}

\begin{rem}
The piecewise H\"{o}lder regularity of the gradient cannot be improved.
Indeed, by the counterexamples of Nadirashvili-Vl\u adu\c t~\cite{NV1,NV2}, we can find a uniformly elliptic fully nonlinear operator $F$ and a viscosity solution $u \in C(\overline{B}_1)$ to the equation $F(D^2 u) = 0$ in $B_1$, such that $u$ is $C^{1,\alpha}$ but not $C^{1,1}$ at the origin.
It follows that $u$ is a viscosity solution to the transmission problem
\[
\begin{cases}
F(D^2 u) = 0 & \text{ in } B_1^{\pm},\\
u_{x_n}^{+}- u_{x_n}^{-} = 0 & \text{ on } T_1,\\
\end{cases}
\]
hence we cannot extend the estimates in Theorem~\ref{thm:piecewise} to include $C^{1,1}$ norms.
Moreover, the gradient will be discontinuous in general, as can be seen from the simple transmission condition 
\[
u_{x_n}^{+} - 2 u_{x_n}^{-} = 0.
\]
It is not clear what the optimal H\"{o}lder exponent might be.
In principle, it could be smaller than the best exponent for second order fully nonlinear equations.
\end{rem}

\medskip

To prove Theorem~\ref{thm:piecewise}, we exploit the invariance of constant-coefficient transmission problems under horizontal translations.
This property allows us to show that tangential difference quotients (and, ultimately, the tangential derivatives) satisfy the H\"{o}lder estimate in Theorem~\ref{thm:holder:intro}.
We then combine this fact with the classical boundary regularity for the Dirichlet problem on each phase to obtain a $C^{1,\alpha}$ estimate up to the interface, from each side (we give the details in  Section~\ref{sec:piecewise}).

\medskip

Our third and final result establishes piecewise $C^{1,\alpha}$ regularity for variable-coefficient inhomogeneous problems.
Here, we need the stronger notion of \emph{viscosity* solution} defined with respect to a larger class of test functions (see Definition~\ref{def:viscstar} below).
With a slight abuse of notation, we define the $C^{\alpha-1}$-Morrey norm of $f^{\pm} \in C(B_R^{\pm} \cup T_R)$ by
\begin{equation}
\label{morrey}
\|f^{\pm}\|_{C^{\alpha-1}(\overline{B}^{\pm}_R)} := \sup_{x \in \R^n}\sup_{B_r(x) \subset B_R} r^{-\alpha} \|f^{\pm}\|_{L^{n}(B^{\pm}_r(x))},
\end{equation}
where $B_r(x)^{+} := B_r(x) \cap \{x_n > 0\}$ and $B_r(x)^{-} := B_r(x) \cap \{x_n < 0\}$ are not necessarily half-balls.

\begin{thm}
[Piecewise $C^{1,\alpha}$ regularity for variable coefficients]
\label{thm:main}
Let $\{F^{\pm}, G\}$ be H\"{o}lder continuous operators in $B_1^{\pm} \cup T_1$ and $T_1$, respectively, and with ellipticity constants $0 < \lambda \leq \Lambda$ and $\mu \geq 0$.
Let $\{f^{\pm}, g\}$ be H\"{o}lder continuous functions in $B_1^{\pm} \cup T_1$ and $T_1$, respectively.
Let $0 \leq \theta \leq 1$ be a constant.
Let $u \in C(B_1)$ be a viscosity* solution to the transmission problem
\[
\begin{cases}
F^{\pm}(D^2u, x) = f^{\pm}(x) & \text{ in } B_1^{\pm},\\
G(\nabla u^{+}, \theta \nabla u^{-}, x) =g(x) & \text{ on } T_1.
\end{cases}
\]
Assume that $F^{\pm}(0,x) = 0$ for $x \in B_1^{\pm}$.

If $\{G, g\}$ are $C^{\alpha}$ on $\overline{T}_1$ for some $0 <\alpha < \overline{\alpha}$, with  $\overline{\alpha}$ given by Theorem~\ref{thm:piecewise} above, then the restrictions $u^{\pm}$ are $C^{1,\alpha}$ on $B_{1}^{\pm} \cup T_1$ and satisfy
\[
\begin{split}
\|u^{\pm}\|_{C^{1,\alpha}(\overline{B}^{\pm}_{1/2})}
\leq C \left(\|u\|_{L^{\infty}(B_1)} 
+ \|G(0,0,\cdot)\|_{L^{\infty}(T_1)}
+ \|f^{+}\|_{C^{\alpha-1}(\overline{B}_1^{+})}
+ \|f^{-}\|_{C^{\alpha-1}(\overline{B}_1^{-})}
+ \|g\|_{C^{\alpha}(\overline{T}_1)}
 \right)
 \end{split}
\]
for some constant $C > 0$ depending only on $n$, $\lambda$, $\Lambda$, $\mu$, $\alpha$, $\overline{\alpha}$, the modulus of continuity of $F^{\pm}$ on $\overline{B}_{1}^{\pm}$, and the $C^{\alpha}$ seminorm of $G$ on $\overline{T_1}$.
Moreover, $u$ satisfies the transmission condition in the classical sense.
\end{thm}

\begin{rem}
The theorem continues to hold without the assumption $F^{\pm}(0,x) = 0$, replacing the functions $f^{\pm}$ in the estimate by $f^{\pm} - F^{\pm}(0, \cdot)$.
\end{rem}

The proof of Theorem~\ref{thm:main} is based on a delicate perturbative argument around constant-coefficient homogeneous transmission problems inspired by Caffarelli's work on fully nonlinear equations~\cite{C1,C2}.
Here, the key step involves showing that solutions are close to appropriate replacements satisfying such problems.
To prove this, we argue by contradiction-compactness, assuming that the approximation fails along a sequence of solutions, from which we can then extract a converging subsequence.
Compactness is a consequence of the interior H\"{o}lder estimate in Theorem~\ref{thm:holder:intro} combined with the boundary behavior of solutions, which we obtain in Section~\ref{sec:barrier} via suitable piecewise smooth barriers. 
The crucial point is to show that the limiting function solves a constant-coefficient problem, whereupon our uniqueness result established in~\cite{E} yields a contradiction.
For this, we use non-explicit barriers $\{\psi_k\}$ given as solutions to certain constant-coefficient inhomogeneous transmission problems.
This forces us to first prove piecewise $C^{1,\alpha}$ regularity for such problems (Theorem~\ref{thm:const}), 
which we carry out in Section~\ref{sec:inhom} by perturbation.
We give the full proof of Theorem~\ref{thm:main} in Section~\ref{sec:pert}.

\medskip

It is worth noting that the problems satisfied by the barriers $\{\psi_k\}$ from the proof above are given by convex equations.
If we could show piecewise $C^{1,1}$ regularity for such convex transmission problems (think of the Evans-Krylov $C^{2,\alpha}$ estimate), then Theorem~\ref{thm:main} would continue to hold for ordinary viscosity solutions (see Remark~\ref{rem:convex}).
Whether such an improved regularity holds is an open question that we intend to investigate in the future.

\medskip

\subsection{References and state of the art}

It was Picone~\cite{P} who brought transmission problems to the attention of the mathematical community.
Shortly after, a variational approach was proposed.
Here, one studies \emph{weak} solutions obtained by minimization of heterogeneous energy functionals.
The first variation then yields a divergence form equation with discontinuous coefficients, for which the classical De Giorgi-Nash theory is applicable.
This is both the starting point and the key element of the (by now very well understood) variational regularity theory for transmission problems.
For an account of this vast topic, we refer the reader to the monograph of Borsuk~\cite{B} and the references therein.
A summary of the regularity theory can also be found in the celebrated textbook by Ladyzhenskaya-Uraltseva~\cite[Ch.~3.16]{LU}, where such models are referred to as ``diffraction problems''.

\medskip

A non-variational framework for transmission problems is a much more recent topic that has only emerged in the last decade.
To the best of our knowledge, the first viscosity solution approach to such problems is due to De Silva-Ferrari-Salsa~\cite{DFSfb}.
In that work, the authors studied the regularity of a fully nonlinear free boundary problem.
Upon linearizing at a free boundary point, their analysis led them to a constant-coefficient homogeneous transmission problem with flat interface.
Here, they considered nonlinear bulk equations with a linear Neumann transmission condition, i.e., a coupling of normal derivatives across the interface.
The authors obtained piecewise $C^{1,\alpha}$ estimates for viscosity solutions, later extending this result to equations involving bulk source terms in~\cite{DFS}.

\medskip

The only other paper investigating  the regularity of a non-variational transmission problem is by Soria Carro-Stinga~\cite{SS}.
There, the authors proved piecewise $C^{1,\alpha}$ estimates for a constant-coefficient problem with curved interface, including source terms in both the bulk equations and the transmission condition.
As in all previous results, they only treat linear Neumann transmission conditions.
We also mention the recent work of Jesus-Soria Carro~\cite{JS}, where this result has been extended to the parabolic setting.

\medskip

Our contribution in the present paper is to extend nearly every previous regularity result to nonlinear, variable-coefficient oblique transmission conditions, a subject that had not been considered until our recent work~\cite{E}.
This task is quite involved as the implicit use of linearity, constant coefficients, and normal derivatives greatly simplified the analysis in~\cite{DFSfb,DFS,SS}.
The key qualitative difference between the two settings can be recognized in the barriers involved in the proofs.
Most notably, to prove the Harnack inequality in the Neumann setting, it suffices to consider simple radial functions that do not behave differently across the interface.
By contrast, we need to carefully construct piecewise quadratic barriers to account for both the tangential contribution of the gradient and the relative size of its normal components from each side of the interface.

\medskip

We treat curved interfaces in a forthcoming work.
A crucial tool in the Neumann setting of~\cite{SS}, where such surfaces have already been considered, is the use of certain piecewise smooth barriers obtained as solutions to the classical Dirichlet problem for fully nonlinear equations.
This gives a clean proof of both the ABP estimate and the Harnack inequality.
Unfortunately, the same method fails in the variable-coefficient oblique case.
Instead, we apply a flattening procedure to the interface and argue similarly to~\cite{E} and the present paper.
Such a transformation preserves the transmission condition, while the bulk equations involve lower order terms with coefficients blowing up at the interface.
The analysis is more technical, but we reach the same piecewise $C^{1,\alpha}$ estimate.

\medskip

To conclude, we would like to point out that our transmission problems are strongly related to boundary-value problems involving oblique derivatives.
In fact, they can be thought of as extensions of the latter, as setting the interpolation parameter to $\theta = 0$ recovers the oblique derivative problem on one side.
The earliest result on the regularity for nonlinear oblique derivative problems that we are aware of is by Lieberman-Trudinger~\cite{LT}.
Our barriers in Section~\ref{sec:holder} are inspired by that work.
The Neumann problem with fully nonlinear bulk equations was studied by Milakis-Silvestre~\cite{MS}, and later extended to oblique derivative conditions by Li-Zhang~\cite{LZ}.

\subsection{Notation}
\label{sub:notation}

Here we introduce some basic notation used throughout the manuscript.

\medskip

Given an open set $\Omega \subset \R^n$, we denote its bulk phases by
\[
\Omega^{+} = \Omega \cap \{x_n > 0\} \qquad \text{ and } \qquad \Omega^{-} = \Omega \cap \{x_n < 0\},
\]
and its flat interface by
\[
T = \Omega \cap \{x_n = 0\} = \partial \Omega^{\pm} \cap \Omega.
\]
For $x_0 \in \R^{n}$ and $R> 0$, we denote the open balls in $\R^{n}$ by
\[
B_{R}(x_0) = \{x \in \R^n \colon |x - x_0| < R\} \qquad \text{ and } \qquad B_R = B_{R}(0).
\]
When $x_0 \in T$, we also define the open balls on the interface as
\[
T_R(x_0) = B_R(x_0) \cap \{x_n = 0\}\qquad \text{ and }\qquad T_R = T_R(0).
\]

\medskip

For a continuous function $u$ in $\Omega$, we define its restrictions to each phase, up to the interface, by
\[
u^{\pm} := u \chara{\overline{\Omega}^{\pm} \cap \Omega} = u \chara{\Omega^{\pm} \cup T},
\]
where $\chara{E}$ is the characteristic function of the set $E \subset \R^{n}$.
The positive and negative parts of $u$ are
\[
u_{+} := \max\{u, 0\} \qquad \text{and} \qquad u_{-} := \max\{-u, 0\}.
\]
Note our use of superscripts for restrictions, while subscripts denote the positive and negative parts.

\medskip

For matrices $N \in \R^{n\times n}$, we always use the operator norm
\[
\|N\| = \sup_{x \in B_1 \subset \R^{n}} |N x|.
\]
The space of symmetric $n \times n$ matrices is denoted
$\scal_n \subset \R^{n\times n}$.
The set of elliptic matrices $A \in \scal_n$ satisfying the bounds $\lambda \id \leq A \leq \Lambda \id$ is written $\ellmat$.
We denote the Pucci extremal operators acting on matrices $M \in \R^{n\times n}$ by
\[
\mcal^{+}(M; \lambda, \Lambda) = \sup_{A \in \ellmat}  \tr{A M} \qquad \text{ and } \qquad \mcal^{-}(M; \lambda, \Lambda) = \inf_{A \in \ellmat}  \tr{A M}.
\]

\medskip

 We write vectors in terms of tangential and normal components as $x = (x', x_n) \in \R^{n-1} \times \R$.
The set of oblique derivative coefficients with ellipticity constants $0 < \lambda \leq \Lambda$ and $\mu \geq 0$ is given by
\[
\ellvec = 
\left\{
\gamma = (\gamma', \gamma_n) \in \R^{n-1} \times \R \colon \quad \lambda \leq \gamma_n \leq \Lambda, \quad |\gamma'| \leq \mu
\right\}.
\]
We define the extremal transmission operators $\tcal^{\pm}$ acting on vectors $\xi^{\pm} = (\xi'^{\pm}, \xi_n^{\pm}) \in \R^{n-1}\times \R$ by
\[
\begin{split}
\tcal^{+}(\xi^{+}, \xi^{-}) 
&= \tcal^{+}(\xi^{+}, \xi^{-}; \lambda,\Lambda,\mu)
= \sup_{\gamma^{\pm} \in \ellvec} \left[\gamma^{+} \cdot \xi^{+} - \gamma^{-} \cdot \xi^{-}\right],\\
\tcal^{-}(\xi^{+}, \xi^{-})  
&= \tcal^{-}(\xi^{+}, \xi^{-}; \lambda,\Lambda,\mu)
= \inf_{\gamma^{\pm} \in \ellvec} \left[\gamma^{+} \cdot \xi^{+} - \gamma^{-} \cdot \xi^{-}\right].
\end{split}
\]
For reference, we note that these operators can be written explicitly as
\begin{equation}
\label{tminus}
\begin{split}
\tcal^{+}(\xi^{+}, \xi^{-}) &= 
\Lambda (\xi_n^{+})_{+} - \lambda (\xi_n^{+})_{-}
- \lambda (\xi_n^{-})_{+}+ \Lambda (\xi_n^{-})_{-}
+ \mu |\xi'^{+}| + \mu |\xi'^{-}|,\\
\tcal^{-}(\xi^{+}, \xi^{-}) &= 
\lambda (\xi_n^{+})_{+} - \Lambda (\xi_n^{+})_{-}
- \Lambda (\xi_n^{-})_{+}+ \lambda (\xi_n^{-})_{-}
- \mu |\xi'^{+}| - \mu |\xi'^{-}|,
\end{split}
\end{equation}
depending on the sign of the normal components $\xi^{\pm}_{n}$.
The operators $\tcal^{\pm}$ were introduced in~\cite{E}.

\medskip

Derivatives are written as subscripts, e.g., $u_{x_i}$, $u_{x_i x_j}$, and so on.
We also use the gradient $\nabla u = (u_{x_1}, \ldots, u_{x_n})$ vector 
and the tangential gradient $\nabla' u = (u_{x_1}, \ldots, u_{x_{n-1}}, 0)$.

\medskip

\subsection{Outline}
In Section~\ref{sec:visc} we recall the framework of viscosity solutions and some fundamental results from our previous work~\cite{E}.
Section~\ref{sec:holder} establishes a boundary Harnack inequality for functions in the solution class, whence Theorem~\ref{thm:holder:intro} follows.
Section~\ref{sec:piecewise} is devoted to the proof of Theorem~\ref{thm:piecewise}.
In Section~\ref{sec:barrier} we use barriers to prove Theorem~\ref{thm:barriers}, an equicontinuity result crucial to the perturbation theory.
Section~\ref{sec:inhom} is devoted to proving Theorem~\ref{thm:const}, an auxiliary regularity result needed in the proof of the main theorem.
In Section~\ref{sec:pert} we establish Theorem~\ref{thm:main} via perturbation.
Appendix~\ref{app:boundary} contains an accessory lemma on gluing interior and boundary moduli of continuity used in the proof of Theorem~\ref{thm:barriers}.

\medskip

%%%%%%%%%%%%%%%%%%%%%%%%%%%%%%%%%%%%%%%%%

\section{Framework and basic tools}
\label{sec:visc}

Here, we recall the framework of viscosity solutions as well as various fundamental results from our previous paper~\cite{E}.
Given an open set $\Omega \subset \R^n$, we consider the transmission problem
\begin{equation}
\label{eq:nl}
\begin{cases}
F^{\pm} (D^2 u, x) = f^{\pm}(x) & \text{ in } 
\Omega^{\pm},\\
G(\nabla u^{+}, \theta \nabla u^{-}, x) = g(x)
 & \text{ on } T,
\end{cases}
\end{equation}
where the functions $F^{\pm}(M, \cdot)$ and $f^{\pm}$ are continuous in $\Omega^{\pm} \cup T$,
the functions $G(\xi^{+}, \xi^{-}, \cdot)$ and $g$ are continuous on $T$, and $0 \leq \theta \leq 1$ is a constant.
We assume that the operators $F^{\pm}$ and $G$ are uniformly elliptic, in the sense that there are constants $0 < \lambda \leq \Lambda$ and $\mu \geq 0$ such that:
\begin{itemize}
\item For any symmetric matrices $M$ and $N$, with $N$ nonnegative definite, we have
\[ 
\lambda \|N\| \leq F^{\pm}(M + N,x) - F^{\pm}(M,x) \leq \Lambda \|N\|;
\]
\item
For any vectors $\xi^{\pm}$ and $\eta^{\pm} = (\eta'^{\pm}, \eta_n^{\pm})$ in $\R^{n} = \R^{n-1} \times \R$, with $\eta_n^{\pm} \geq 0$, we have
\[
G(\xi^{+} + \eta^{+}, \xi^{-} + \eta^{-}, x) \leq G(\xi^{+}, \xi^{-}, x) + \Lambda \eta^{+}_n - \lambda \eta^{-}_n  + \mu( |\eta'^{+}| + |\eta'^{-}|)
\]
and
\[
G(\xi^{+} + \eta^{+}, \xi^{-} + \eta^{-}, x) \geq 
G(\xi^{+}, \xi^{-}, x) + \lambda \eta^{+}_n - \Lambda \eta^{-}_n 
- \mu (|\eta'^{+}| + |\eta'^{-}|).
\]
\end{itemize}

\medskip

Let $u \in C(\Omega)$.
We say that 
a continuous function $\varphi$ \emph{touches} $u$ \emph{ from above} (resp. \emph{from below}) at a point $x_0 \in \Omega$ if $\varphi(x_0) = u(x_0)$ and 
\[
\varphi(x) \geq u(x) \qquad \text{(resp. } \varphi(x) \leq u(x))
\]
for all $x$ in a neighborhood of $x_0$.
It touches \emph{strictly} if the inequalities are strict for $x \neq x_0$.

\medskip

By a \emph{piecewise quadratic function}, we mean a continuous function $P \in C(\R^{n})$ such that each restriction $P^{\pm}$ is a quadratic polynomial 
$P^{\pm}(x) = \frac{1}{2} M^{\pm} (x-x_0) \cdot (x-x_0) + a^{\pm} \cdot (x-x_0) + c$
for some symmetric matrices $M^{\pm} \in \scal_n$, vectors $a^{\pm} \in \R^{n}$, and constant $c \in \R$.

\medskip

We recall the following notion of viscosity solution introduced in~\cite{E}:
\begin{defn}[Viscosity solution]
\label{def:visc}
A continuous function $u \in C(\Omega)$ is a 
\emph{viscosity subsolution} (resp. \emph{supersolution}) to~\eqref{eq:nl} if whenever a piecewise quadratic function $P$ touches $u$ from above (resp. below) at $x_0 \in \Omega$ and:
\begin{itemize}
\item $x_0 \in \Omega^{\pm}$, then
\[
F^{\pm}(D^2 P^{\pm}, x_0) \geq f^{\pm}(x_0) \qquad (\text{resp. } F^{\pm}(D^2 P^{\pm}, x_0) \leq f^{\pm}(x_0));
\]
\item $x_0 \in T$, then
\[
G(\nabla P^{+}(x_0), \theta \nabla P^{-}(x_0), x_0) \geq g(x_0) 
\qquad 
\text{(resp. } 
G(\nabla P^{+}(x_0), \theta \nabla P^{-}(x_0), x_0) \leq g(x_0) ).
\]
\end{itemize}
We say that $u$ is a \emph{viscosity solution} to~\eqref{eq:nl} if it is both a viscosity subsolution and supersolution.
\end{defn}

\begin{rem}
\label{rem:trick}
When $x_0 \in T$, by a standard device, we may always assume that $P$ touches $u$ strictly and additionally satisfies
\[
F^{\pm}(D^2 P, x_0) < - \vep^{-1} 
\qquad 
\text{(resp. } 
F^{\pm}(D^2 P, x_0) > \vep^{-1}
\text{)}
\]
in the subsolution (resp. supersolution) condition above, with $\vep >0$ as small as we like; see Lemma~2.6 in~\cite{E} for the details.
\end{rem}

\medskip

The extremal operators $\{\mcal^{\pm}, \tcal^{\pm}\}$ introduced in the notation section (see subsection~\ref{sub:notation} above) give rise to universal viscosity subsolution and supersolution classes:
\begin{defn}
We say that $u \in \underline{S}(\lambda,\Lambda,\mu,\theta; f^{+}, f^{-}, g)$\emph{ in }$\Omega$, if $u$ is a viscosity subsolution to 
\[
\begin{cases}
\mcal^{+}(D^2 u; \frac{\lambda}{n}, \Lambda) \geq f^{\pm}(x) & \text{ in } \Omega^{\pm},\\
\tcal^{+}(\nabla u^{+}, \theta \nabla u^{-}; \lambda,\Lambda,\mu) \geq g(x) & \text{ on } T.
\end{cases}
\]
Similarly, we say that $u \in \overline{S}(\lambda,\Lambda,\mu,\theta; f^{+}, f^{-}, g)$\emph{ in }$\Omega$, if $u$ is a viscosity supersolution to
\[
\begin{cases}
\mcal^{-}(D^2 u; \frac{\lambda}{n}, \Lambda) \leq f^{\pm}(x) & \text{ in } \Omega^{\pm},\\
\tcal^{-}(\nabla u^{+}, \theta \nabla u^{-}; \lambda,\Lambda,\mu) \leq g(x) & \text{ on } T.
\end{cases}
\]
We also define the solution classes
\[
S(\lambda,\Lambda,\mu,\theta; f^{+}, f^{-}, g) = \underline{S}(\lambda,\Lambda,\mu,\theta; f^{+}, f^{-}, g) \cap \overline{S}(\lambda,\Lambda,\mu,\theta; f^{+}, f^{-}, g)
\]
and
\[
S^{\star}(\lambda,\Lambda,\mu,\theta; f^{+}, f^{-}, g) = \underline{S}(\lambda,\Lambda,\mu, \theta; -|f^{+}|, - |f^{-}|, -|g|) \cap \overline{S}(\lambda,\Lambda,\mu, \theta; |f^{+}|, |f^{-}|, |g|).
\]
\end{defn}

\medskip

To conclude this section, we list several fundamental results obtained in our previous work~\cite{E}.
First we have the well-posedness of the Dirichlet problem for constant-coefficient transmission problems in balls:

\begin{thm}[Existence and uniqueness -- Theorem~1.1 in~\cite{E}]
\label{thm:wellposed}
Let $\{F^{\pm}, G\}$ be constant-coefficient uniformly elliptic operators.
Let $f^{\pm} \in C(\overline{B}_1^{\pm})$, $g \in C(\overline{T}_1)$, and $\phi \in C(\partial B_1)$.
Let $0 \leq \theta \leq 1$ be a constant.

There exists a unique viscosity solution $u \in C(\overline{B}_1)$ to the transmission problem
\[
\begin{cases}
F^{\pm}(D^2u) = f^{\pm}(x) & \text{ in } B_1^{\pm},\\
G(\nabla u^{+}, \theta \nabla u^{-}) =g(x) & \text{ on } T_1,\\
u = \phi & \text{ on } \partial B_1.
\end{cases}
\]
\end{thm}

\medskip

Next, we recall our ABP maximum principle for functions in the supersolution class, which is a basic tool in the proofs of all our estimates below:

\begin{thm}[ABP -- c.f. Theorem~3.1 in~\cite{E}]
\label{thm:abp}
Let $u \in \overline{S}(\lambda,\Lambda,\mu, \theta; f^{+}, f^{-}, g)$ in $B_R$, and assume that $u \in C(\overline{B_R})$.
There is a universal constant $C > 0$ such that
\[
\sup_{B_R} u_{-} \leq \sup_{\partial B_R} u_{-} + C R \left( 
\|f^{+}_{+}\|_{L^{n}(B_R^{+})} + \|f^{-}_{+}\|_{L^{n}(B_R^{-})} + \|g_{+}\|_{L^{\infty}(T_R)}
 \right).
\]
\end{thm}

\medskip

Finally, we state our comparison principle for constant-coefficient transmission problems:

\begin{thm}[Comparison principle -- c.f. Theorem~4.3 in~\cite{E}]
\label{thm:comp}
Let $\{F^{\pm},G\}$ be constant-coefficient operators with ellipticity constants $0 < \lambda \leq \Lambda$ and $\mu \geq 0$.
Let $\{f^{\pm}, \widetilde{f}^{\pm}\}$ be functions in $C(B^{\pm}_R \cup T_R)$, let $\{g, \widetilde{g}\}$ be functions in $C(T_R)$, and let $0 \leq \theta \leq 1$ be a constant.

If $u \in C(B_R)$ and $v \in C(B_R)$ are a viscosity subsolution and supersolution to 
\[
\begin{cases}
F^{\pm}(D^2 u) \geq f^{\pm}(x) & \text{ in } B_R^{\pm},\\
G(\nabla u^{+}, \theta \nabla u^{-}) \geq g(x) & \text{ on } T_R,
\end{cases}
\quad \text{ and } \quad
\begin{cases}
F^{\pm}(D^2 v) \leq \widetilde{f}^{\pm}(x) & \text{ in } B_R^{\pm},\\
G(\nabla v^{+}, \theta \nabla v^{-}) \leq \widetilde{g}(x) & \text{ on } T_R,
\end{cases}
\]
respectively, then
\[
u - v \in \underline{S}(\lambda,\Lambda,\mu,\theta; f^{+} - \widetilde{f}^{+}, f^{-} - \widetilde{f}^{-}, g - \widetilde{g}) \quad \text{ in } B_R.
\]
\end{thm}

\medskip

\section{Harnack inequality and H\"{o}lder estimate}
\label{sec:holder}

Our main result in this section is a boundary (or rather an \emph{interface}) Harnack inequality for functions in the solution class.
To prove it, we modify the classical quadratic barrier for the oblique derivative problem~\cite{LT} by making it piecewise quadratic, and then combine it with the ABP estimate (Theorem~\ref{thm:abp}).

\medskip

To state the inequality accurately, it is convenient to introduce a special notation for thin cylinders.
For $\rho^{\pm} > 0$ and $R > 0$, we define the two-sided cylinder
\[
\gcal(\rho^{-}, \rho^{+}, R) = B_R' \times (- R \rho^{-}, R \rho^{+}) \subset \R^{n-1} \times \R
\]
and the one-sided cylinder away from the interface
\[
\widetilde{\gcal}^{+}(\rho^{+}, R) 
= B_R' \times \textstyle (\frac{1}{2}R \rho^{+}, R \rho^{+}).
\]
We also write $Q_R$ for the standard ``thick'' two-sided cylinder
\[
Q_R = \gcal(1,1, R) = B_R' \times(-R,R).
\]

\medskip

\begin{lem}[Harnack inequality]
\label{lem:harnack}
Let $u \in S^{\star}(\lambda, \Lambda,\mu,\theta; f^{+}, f^{-},g)$ in $Q_3$.

There are universal constants 
$0 < \rho^{-} < 1$, $\rho^{+} = \frac{\lambda}{4\Lambda} \rho^{-}$, 
and $C > 0$ such that if
\[
u \geq 0 \quad \text{ in } \gcal(\rho^{-}, \rho^{+}, 3),
\]
then
\[
\sup_{\widetilde{\gcal}^{+}(\rho^{+}, 2)} u \leq C \left( \inf_{\gcal(\rho^{-}, \rho^{+}, 1)} u
+\|f^{+}\|_{L^{n}(Q_3^{+})} + \|f^{-}\|_{L^{n}(Q_3^{-})}+ \|g\|_{L^{\infty}(T_3)}
\right).
\]
\end{lem}

\begin{proof}
Once we have chosen $\rho^{\pm}$, the classical Harnack inequality in $\gcal^{+}(\rho^{-},\rho^{+},3) = B_3' \times (0, 3\rho^{+})$ (Theorem~4.3 in~\cite{CC}) and a chaining argument yield
\begin{equation}
\label{into}
\sup_{\widetilde{\gcal}^{+}(\rho^{+},2)} u
\leq C \left( \inf_{\widetilde{\gcal}^{+}(\rho^{+},2)} u + \|f^{+}\|_{L^{n}(Q^{+}_3)} \right)
\end{equation}
for some universal constant $C > 0$. 
Hence it suffices to control the infimum in $\widetilde{\gcal}^{+}(\rho^{+},2)$.

\medskip

Let $m = \inf_{\widetilde{\gcal}^{+}(\rho^{+}, 2)} u$ and
consider the piecewise quadratic function
\[
\varphi = \frac{1}{2} 
- \frac{1}{4} |x'|^2
+\frac{1}{4} \left[
\frac{x_n}{2\rho^{+}} 
+ \left( \frac{x_n}{2\rho^{+}}\right)^2
\right] \chara{\{x_n \geq 0\}}
+ \frac{1}{4} \left[
\frac{x_n}{ \rho^{-}}
+  \left( \frac{x_n}{3\rho^{-}}\right)^2
\right] \chara{\{x_n \leq 0\}}.
\]
We will show that the barrier $m\varphi$ controls $u$ from below in the region
\[
\gcal_2 := 
\textstyle\gcal(\frac{3}{2}\rho^{-}, \rho^{+}, 2) 
= B_2' \times (- 3 \rho^{-}, 2 \rho^{+}),
\]
for an appropriate choice of $\rho^{-}$ (recall that $\rho^{+} = \frac{\lambda}{4 \Lambda} \rho^{-}$).

\medskip

First, we claim that at the boundary of $\gcal_2$, we have the upper bounds
\begin{equation}
\label{bedo1}
\varphi \leq 1 \quad \text{ on } \partial \gcal_2 \cap \{x_n = 2\rho^{+}\}
\end{equation}
and
\begin{equation}
\label{bedo2}
\varphi \leq 0 \quad \text{ on } \partial \gcal_2 \setminus \{x_n = 2\rho^{+}\}.
\end{equation}
Indeed, direct computations show that:
\begin{itemize}
\item At the top lid $\partial \gcal_2 \cap \{x_n = 2\rho^{+}\}$
\[
\varphi \leq \frac{1}{2} + \frac{1}{4} (1 + 1) = 1;
\]
\item At the upper lateral boundary 
$\{|x'| = 2\} \cap \{0 \leq x_n \leq 2 \rho^{+}\}$
\[
\varphi \leq \frac{1}{2} + \frac{1}{4}(1 + 1- 2^2) = 0;
\]
\item At the lower lateral boundary 
$\{|x'| = 2\} \cap \{-3\rho^{-} \leq x_n \leq 0 \}$
\[
\varphi \leq \frac{1}{2} + \frac{1}{4}(-2^2) = - \frac{1}{2} <0;
\]
\item At the bottom lid $\partial \gcal_2 \cap \{x_n = - 3\rho^{-}\}$ 
\[
\varphi \leq \frac{1}{2} + \frac{1}{4} (-3+1) = 0.
\]
\end{itemize}
Moreover, 
the Hessian for $x_n < 0$ is
\[
D^2 \varphi = 
\frac{1}{2}
\left[
\frac{1}{( 3\rho^{-})^2} e_n \otimes e_n - \id'
\right]
\]
(here, with a slight abuse of notation $I' = \id - e_n \otimes e_n$)
and hence
\[
\textstyle
\mcal^{-}(D^2 \varphi; \frac{\lambda}{n}, \Lambda) = 
\displaystyle
\frac{1}{2} \left(\frac{\lambda}{n(3 \rho^{-})^2} - \Lambda(n-1) \right) 
\quad \text{ for } x_n < 0,
\]
which is positive for
\begin{equation}
\label{into2}
\rho^{-} < \frac{1}{3} \sqrt{\frac{\lambda}{n(n-1)\Lambda}}.
\end{equation}
Recalling that $\rho^{+} = \frac{\lambda}{4 \Lambda} \rho^{-} < \rho^{-}$, 
under assumption~\eqref{into2},
it is immediate to check that we also have
\begin{equation}
\label{bedo3}
\textstyle
\mcal^{-}(D^2 \varphi; \frac{\lambda}{n}, \Lambda) > 0 \quad \text{ for } x_n > 0 \text{ and } x_n < 0.
\end{equation}

\medskip

Since $\varphi_{x_n}^{+} = \frac{1}{8 \rho^{+}} >0$ and $\varphi_{x_n}^{-} = \frac{1}{4 \rho^{-}} >0$,  by~\eqref{tminus}, the transmission condition on $T_{2}$ is
\[
\begin{split}
\tcal^{-}(\nabla \varphi^{+},\theta \nabla \varphi^{-}; \lambda, \Lambda, \mu) 
& = \lambda \varphi_{x_n}^{+} - \Lambda \theta \varphi_{x_n}^{-} - (1+ \theta) \mu |\nabla' \varphi|\\
&= \frac{1}{4} \left[
\frac{\lambda}{2\rho^{+}}
- \frac{\Lambda}{\rho^{-}} \theta - 4(1+ \theta)\mu |x'|
\right]\\
& \qquad \geq \frac{1}{4} \left[
\frac{\lambda}{2\rho^{+}}
- \frac{\Lambda}{\rho^{-}} - 8 \mu
\right]\\
& \qquad \quad \geq \frac{\Lambda}{4 \rho^{-}}- 2 \mu,
\end{split}
\]
therefore, assuming that
\begin{equation}
\label{into3}
\rho^{-} 
< \frac{\Lambda}{8 \mu},
\end{equation}
the extremal operator becomes positive, i.e.,
\begin{equation}
\label{bedo4}
\tcal^{-}(\nabla \varphi^{+},\theta \nabla \varphi^{-};\lambda, \Lambda, \mu)  > 0 \quad \text{ on } \{x_n = 0\}.
\end{equation}

\medskip

Combining \eqref{into2} and \eqref{into3}, we finally take the universal constant
\[
\rho^{-} := \frac{1}{2} \min\left\{
\frac{1}{3} \sqrt{\frac{\lambda}{n(n-1)\Lambda}},
\frac{\Lambda}{8 \mu}
\right\}.
\]
With this choice, from~\eqref{bedo1}, \eqref{bedo2}, \eqref{bedo3}, and~\eqref{bedo4} it follows that
\[
\begin{cases}
u - m \varphi \in \overline{S}(\lambda,\Lambda,\mu,\theta; |f^{+}|, |f^{-}|, |g|) & \text{ in } \gcal_2,\\
u - m \varphi \geq 0 & \text{ on } \partial \gcal_2,
\end{cases}
\]
whence by the ABP maximum principle, we deduce
\begin{equation}
\label{bedo5}
m \varphi \leq u + 
C\left( 
\|f^{+}\|_{L^n(Q^{+}_3)} + \|f^{-}\|_{L^n(Q^{-}_3)}
+ \|g\|_{L^{\infty}(T_3)} \right) 
\quad \text{ in } \gcal_2.
\end{equation}

\medskip

Moreover, the barrier $\varphi$ is positive in $\gcal(\rho^{-}, \rho^{+}, 1)$, with a uniform bound.
Namely:
\begin{itemize}
\item 
In $\gcal(\rho^{-}, \rho^{+}, 1)^{+}= B_1' \times (0, \rho^{+})$
\[
\varphi \geq \frac{1}{2} - \frac{1}{4} = \frac{1}{4};
\]
\item 
In $\gcal(\rho^{-}, \rho^{+}, 1)^{-}= B_1' \times (-\rho^{-}, 0)$
\[
\varphi \geq 
\frac{1}{2} - \frac{1}{4} + \frac{1}{4} \left[ - 1 + \frac{1}{3^2}\right]
= \frac{1}{36}.
\]
\end{itemize}
Therefore, taking the infimum in $\gcal(\rho^{-},\rho^{+},1)$ in~\eqref{bedo5}, we conclude
\begin{equation}
\label{kryl}
\inf_{\widetilde{\gcal}^{+}(\rho^{+},2)} u = m \leq 36 \inf_{\gcal(\rho^{-},\rho^{+},1)} u + C\left( 
\|f^{+}\|_{L^n(Q^{+}_3)} + \|f^{-}\|_{L^n(Q^{-}_3)}
+ \|g\|_{L^{\infty}(T_3)} \right),
\end{equation}
which combined with~\eqref{into} yields the claim.
\end{proof}
\begin{rem}
It is clear from the proof above that the weaker estimate~\eqref{kryl} holds assuming only that $u \geq 0$ and $u \in \overline{S}(\lambda,\Lambda,\mu,\theta; |f^{+}|, |f^{-}|, |g|)$ in $Q_3$.
\end{rem}

\medskip

As a direct consequence of the boundary Harnack inequality, we deduce H\"{o}lder continuity:

\begin{thm}[$C^{\alpha}$ estimate for functions in $S^{\star}$]
\label{thm:holder}
Let $u \in S^{\star}(\lambda, \Lambda,\mu,\theta; f^{+}, f^{-}, g)$ in $Q_1$.

There are universal constants $0 < \alpha < 1$ and $C >0$ such that
\[
[u]_{C^{\alpha}(\overline{Q}_{1/2})}
\leq C \left( \|u\|_{L^{\infty}(Q_1)} + 
\|f^{+}\|_{L^n(Q^{+}_1)} + \|f^{-}\|_{L^n(Q^{-}_1)}
+ \|g\|_{L^{\infty}(T_1)}\right).
\]
\end{thm}
\begin{proof}
By classical interior regularity (Proposition~4.10 in~\cite{CC}) and a covering argument, it suffices to prove the H\"{o}lder estimate only at the interface $T_{1/2}$. 
Namely, we need to show that
\begin{equation}
\label{eq:holder}
|u(x) - u(x_0)| \leq C \left( \|u\|_{L^{\infty}(Q_1)} 
+\|f^{+}\|_{L^n(Q^{+}_1)} + \|f^{-}\|_{L^n(Q^{-}_1)}
+ \|g\|_{L^{\infty}(T_1)}\right) |x-x_0|^{\alpha}
\end{equation}
for all $x_0 \in T_{1/2}$ and all $x \in Q_1$ sufficiently close to $x_0$.

\medskip

For $0 < R \leq 1$, with the notation above, define
\[
m_R := \inf_{\gcal(\rho^{-}, \rho^{+}, R)} u, 
\qquad M_R := \sup_{ \gcal(\rho^{-}, \rho^{+}, R)} u, 
\qquad o_R := \underset{\gcal(\rho^{-}, \rho^{+}, R)}{\osc} u = M_R - m_R.
\]

\medskip

Let $0 < R \leq 1/6$.
Since $u - m_{3R}$ and $M_{3R} - u$ are nonnegative in $\gcal(\rho^{-}, \rho^{+}, 3R) \subset Q_{1/2}$, applying Lemma~\ref{lem:harnack} to each function (rescaled to $Q_{3R}$) and adding both estimates, we deduce
\[
\begin{split}
M_{3R} -m_{3R} 
& \leq \sup_{\widetilde{\gcal}^{+}(\rho^{+}, 2R)} ( M_{3R} - u) + \sup_{\widetilde{\gcal}^{+}(\rho^{+}, 2R)} (u - m_{3R})\\
& \quad \leq C \left( 
M_{3R} - M_{R} +
m_{R} - m_{3R} 
+ R \|f^{+}\|_{L^n(Q_{3R}^{+})} + R\|f^{-}\|_{L^n(Q_{3R}^{-})}
+ R \|g\|_{L^{\infty}(T_{3R})}\right).
\end{split}
\]
From here, reordering terms, we obtain
\[
o_R \leq \frac{C-1}{C} o_{3R} + R \big( 
\|f^{+}\|_{L^n(Q_{3R}^{+})} + \|f^{-}\|_{L^n(Q_{3R}^{-})}
+ \|g\|_{L^{\infty}(T_{3R})}\big)
\]
and a standard iteration of this inequality leads to an algebraic oscillation decay
\begin{equation}
\label{decco}
o_{r} \leq C r^{\alpha}
\left(o_{1/2}+ \|f^{+}\|_{L^n(Q_{1/2}^{+})} + \|f^{-}\|_{L^n(Q_{1/2}^{-})}+ \|g\|_{L^{\infty}(T_{1/2})}\right) 
\quad \text{ for } r \leq 1/2,
\end{equation}
with $\alpha = \log_{3} \frac{C}{C-1} >0$.
For $x_0 \in T_{1/2}$ and $r \leq 1/2$, we have that
\[
B_{\rho^{+}r}(x_0) \subset x_0 + \gcal(\rho^{-}, \rho^{+}, r) \subset Q_1
\]
and hence a translation of~\eqref{decco} yields
\[
\underset{B_{\rho^{+} r}(x_0)}{\osc} u \leq C r^{\alpha} \left(\|u\|_{L^{\infty}(Q_1)}+ \|f^{+}\|_{L^n(Q_1^{+})} + \|f^{-}\|_{L^n(Q_1^{-})}+ \|g\|_{L^{\infty}(T_{1})} \right),
\]
whence it follows that~\eqref{eq:holder} holds for all $x \in B_{\rho^{+}/2}(x_0)$.
\end{proof}

\medskip

The $C^{\alpha}$ estimate for nonlinear transmission problems, Theorem~\ref{thm:holder:intro} in the Introduction, follows immediately from Theorem~\ref{thm:holder}:
\begin{proof}[Proof of Theorem~\ref{thm:holder:intro}]
Since $u$ belongs to $S(\lambda, \Lambda, \mu, \theta; f^{+} - F^{+}(0,\cdot), f^{-} - F^{-}(0,\cdot), g - G(0,0,\cdot))$ in $B_1$ (see Proposition~2.16 in~\cite{E}), by Theorem~\ref{thm:holder} (rescaled and translated) it satisfies a $C^{\alpha}$ estimate in interior cubes.
The claim follows by a standard covering argument.
\end{proof}

\medskip

We note that combining the classical interior estimates with Theorem~\ref{thm:holder} gives a $C^{\alpha}$ estimate on arbitrary balls, not necessarily centered on the interface:

\begin{cor}
\label{cor:trueint}
Let $u \in S^{\star}(\lambda, \Lambda,\mu,\theta; f^{+}, f^{-}, g)$ in $B_1$.
 
There are universal constants $0 < \alpha < 1$ and $C >0$ such that
\[
R^{\alpha} [u]_{C^{\alpha}(\overline{B}_{R/2}(x))}
\leq 
C \left( \|u\|_{L^{\infty}(B_R(x))} 
+
R \Big(
 \|f^{+}\|_{L^n(B^{+}_R(x))} + \|f^{-}\|_{L^n(B^{-}_R(x))}
+ \|g\|_{L^{\infty}(T \cap B_R(x))}
\Big)
\right)
\]
for every ball $B_R(x) \subset B_1$.
\end{cor}
\begin{proof}
Combining interior estimates and Theorem~\ref{thm:holder}, by a standard scaling and covering argument, the solution $u$ satisfies
\begin{equation}
\label{eq:auxilio}
[u]_{C^{\alpha}(\overline{B}_{1/2})} \leq
C \left(
\|u\|_{L^{\infty}(B_1)} + \|f^{+}\|_{L^n(B^{+}_1)} + \|f^{-}\|_{L^n(B^{-}_1)}+ \|g\|_{L^{\infty}(T_1)}
\right)
\end{equation}
where $C > 0$ is some universal constant.

\medskip

Take a dimensional number of points $x_k \in B_{R/2}(x)$ such that 
$B_{R/2}(x) \subset \cup_k B_{R/16}(x_k)$.
Keep the interior balls $\{B_{R/16}(x_i)\}_{i}$ with 
$B_{R/8}(x_i)\subset \R^n \setminus \{x_n = 0\}$,
and note that $B_{R/8}(x_i) \subset B_{R}(x)$ by assumption.
For the remaining balls $\{B_{R/16}(x_j)\}_{j}$, 
we can find points $y_j \in T \cap B_{R/8}(x_j)$, whence $B_{R/16}(x_j) \subset B_{3R/16}(y_j)$ and also $B_{3R/8}(y_j) \subset B_{R/4}(x_j) \subset B_{R}(x)$ by assumption.

 \medskip

It follows that we can cover $B_{R/2}(x)$ by the families $\{B_{R/16}(x_i)\}_{i}$ and $\{B_{3R/16}(y_j)\}_{j}$.
For the first family, by interior estimates we have that
\[
\begin{split}
 R^{\alpha} [u]_{C^{\alpha}(\overline{B}_{R/16}(x_i))} \leq 
 C \left( \|u\|_{L^{\infty}(B_{R/8}(x_i))} + R \|f\|_{L^{n}(B_{R/8}(x_i))}\right),
\end{split}
\]
and for the second family, by~\eqref{eq:auxilio} (translated and rescaled), we deduce
\[
\begin{split}
&R^{\alpha} [u]_{C^{\alpha}(\overline{B}_{3R/16}(y_i))} \\
& \qquad \leq
C \left( \|u\|_{L^{\infty}(B_{3 R/8}(y_i))} 
+ R \|f^{+}\|_{L^{n}(B^{+}_{3R/8}(y_i))}
 + R \|f^{-}\|_{L^{n}(B^{-}_{3R/8}(y_i))}
 +  R\|g\|_{L^{\infty}(T_{3R/8}(y_i))}
\right).\\
\end{split}
\]
Adding all estimates, we obtain the desired bound on $\overline{B}_{R/2}(x)$.
\end{proof}

%%%%%%%%%%%%%%%%%%%%%%%%%%%%%%%%%%%%%%%%%

\section{Regularity of constant-coefficient homogeneous problems}
\label{sec:piecewise}

We are now in a position to prove Theorem~\ref{thm:piecewise}, the piecewise $C^{1,\overline{\alpha}}$ estimate for constant-coefficient homogeneous transmission problems.
The reasoning is classical and goes as follows.
Since horizontal difference quotients belong to the solution class (see Lemma~\ref{lem:diff} below), 
they satisfy a $C^{\alpha}$ estimate (Section~\ref{sec:holder}).
Iterated, this result yields $C^{1,\alpha}$ regularity in the tangential variables.
In particular, the trace $u|_{T}$ is a $C^{1,\alpha}$ function whence, by boundary regularity on each side, it follows that the restrictions $u^{\pm}$ are $C^{1,\overline{\alpha}}$ up to the interface, for some $\overline{\alpha} \leq \alpha$.

\medskip

\begin{proof}[Proof of~Theorem~\ref{thm:piecewise}]
It will suffice to prove a variant of the estimate in the unit cylinder.
Thus, we first assume that $u \in C(Q_1)$ is a viscosity solution to
\begin{equation}
\label{eq:cubepb}
\begin{cases}
F^{\pm}(D^2 u) = 0 & \text{ in } Q_1^{\pm}\\
G(\nabla u^{+}, \theta \nabla u^{-}) = 0 & \text{ in } T_1,\\
\end{cases}
\end{equation}
where, using the notation from Section~\ref{sec:holder}, we write $Q_R = B_R' \times (-R, R) \subset \R^{n-1} \times \R$ for $R > 0$.

\medskip

The reasoning is similar to the proof of Corollary~5.7 in~\cite{CC}.
From~\eqref{eq:cubepb}, we have that
\[
u \in S(\lambda, \Lambda, \mu, \theta; - F^{+}(0),- F^{-}(0),- G(0,0)) \quad \text{ in } Q_1
\]
(see Proposition~2.16 in~\cite{E}), and hence by the $C^{\alpha}$ estimate (Theorem~\ref{thm:holder}) we deduce
\begin{equation}
\label{eq:har}
\|u\|_{C^{\alpha}(\overline{Q}_{1/2})} \leq C \left(\|u\|_{L^{\infty}(Q_1)} + |F^{+}(0)| + |F^{-}(0)| + |G(0,0)| \right) =: C K,
\end{equation}
where $K := \|u\|_{L^{\infty}(Q_1)} + |F^{+}(0)| + |F^{-}(0)| + |G(0,0)|$.
Here and throughout the proof, we write $C$ to denote a positive universal constant that may change from line to line.

\medskip
First, we note the following corollary of the comparison principle (Theorem~\ref{thm:comp}) concerning horizontal difference quotients:
\begin{lem}[Corollary~4.6 in~\cite{E}]
\label{lem:diff}
Let $u \in C(Q_R)$ be a viscosity solution to
\[
\begin{cases}
F^{\pm}(D^2 u) = 0 & \text{ in } Q^{\pm}_R,\\
G(\nabla u^{+}, \theta \nabla u^{-}) = 0 & \text{ on } T_R.
\end{cases}
\]
Then, for all horizontal vectors $e =(e', 0) \in \R^{n-1} \times \R$, we have that
\[
u(x + e) - u(x) \in S(\lambda,\Lambda,\mu,\theta; 0,0, 0) \quad \text{ in } Q_{R-|e|}.
\]
\end{lem}

\medskip

Let $e$ be a horizontal unit vector $e = (e',0) \in \R^{n-1}\times \R$, with $|e| =1$.
For $0 < h < 1$ and $0 < \beta \leq 1$, we define $v_{\beta, h}$ to be the tangential difference quotient
\[
v_{\beta, h} :=
\frac{u(\cdot + h e) - u}{h^{\beta}}.
\]
By Lemma~\ref{lem:diff}, we have that $v_{\beta, h} \in S(\lambda, \Lambda, \mu, \theta; 0,0,0)$ in $Q_{1- h}$, hence by Theorem~\ref{thm:holder} rescaled
\begin{equation}
\label{eq:har2}
\|v_{\beta,h}\|_{C^{\alpha}(\overline{Q}_r)} \leq C r^{-\alpha} \|v_{\beta,h}\|_{L^{\infty}(Q_{2r})} 
\leq C r^{-\alpha} \sup_{|x_n| < 4 r} \|u(\cdot, x_n)\|_{C^{\beta}(\overline{B}'_{4r})},
\end{equation}
for $0 < r \leq 1/8$ and $h < 2r$.

\medskip

Making $\alpha$ smaller, we may assume that there is an integer $k$ such that $(k-1) \alpha < 1$ and $k \alpha > 1$.
Applying~\eqref{eq:har2} with $\beta = \alpha$ and $r = 1/8$, by~\eqref{eq:har} we obtain
\[
\|v_{\alpha,h}\|_{C^{\alpha}(\overline{Q}_{1/8})} \leq C \|u\|_{C^{\alpha}(\overline{Q}_{1/2})} \leq C K,
\]
where $h < 1/4$.
Since the horizontal vector $e$ was arbitrary, it follows that $u$ is $C^{2\alpha}$ in the horizontal direction, with the estimate
\[
\sup_{|x_n| < 1/8}\|u(\cdot,x_n)\|_{C^{2\alpha}(\overline{B}'_{1/8})}
\leq C K.
\]
Applying~\eqref{eq:har2} iteratively with $\beta = j \alpha$ and $2 \leq j < k-1$, 
we deduce that $u$ is $C^{(k-1) \alpha}$ in the horizontal direction, with an estimate.
Using~\eqref{eq:har2} again, we see that $u$ is Lipschitz continuous in the horizontal variable, with
\[
\sup_{|x_n| < 2^{-2k+1}} \|u(\cdot, x_n)\|_{C^{0,1}(\overline{B}'_{2^{-2k+1}})} \leq C 8^{-(k-1)\alpha} K,
\]
and by~\eqref{eq:har2} with $\beta = 1$ we obtain
\begin{equation}
\label{eq:theold}
\|v_{1,h}\|_{C^{\alpha}(\overline{Q}_{2^{-2k-1}})} \leq C 8^{-k\alpha}K.
\end{equation}

\medskip

Letting $h \downarrow 0$ in~\eqref{eq:theold}, 
we deduce that $u$ is differentiable in the horizontal direction and
\[
\|u_{x_i}\|_{C^{\alpha}(\overline{Q}_{r_k})} \leq C_k K \qquad \text{ for } i < n,
\]
where $r_k := 2^{-2k-1}$ and $C_k > 0$ denotes a generic constant depending only on $n$, $\lambda$, $\Lambda$, $\mu$, and $k$.
In particular, the trace $u|_{T}$ is $C^{1,\alpha}$ on $\overline{T}_{r_k}$, with an estimate
\begin{equation}
\label{eq:har3}
\|u|_{T}\|_{C^{1,\alpha}(\overline{T}_{r_k})} \leq C_k K.
\end{equation}
Since the restrictions $u^{\pm} \in C(Q_{r_k}^{\pm} \cup T_{r_k})$  solve the Dirichlet problems
\[
\begin{cases}
F^{\pm}(D^2 u^{\pm}) = 0 & \text{ in } Q_{r_k}^{\pm},\\
u^{\pm} = u|_{T}&  \text{ on } T_{r_k},\\
\end{cases}
\]
boundary regularity (for instance, see Proposition~2.2 in~\cite{MS}) and \eqref{eq:har3} finally give the estimate
\begin{equation}
\label{eq:theone}
\|u^{\pm}\|_{C^{1,\overline{\alpha}}(\overline{Q}^{\pm}_{r_k/2})} \leq C_k \left(
\|u|_{T}\|_{C^{1,\alpha}(\overline{T}_{r_k})} + \|u^{\pm}\|_{L^{\infty}(Q^{\pm}_{r_k})} + |F^{\pm}(0)| \right) \leq C_k K,
\end{equation}
with $\overline{\alpha} > 0$ universal and possibly smaller.

\medskip

Therefore, viscosity solutions to~\eqref{eq:cubepb} satisfy the piecewise $C^{1,\overline{\alpha}}$ estimate~\eqref{eq:theone}.
A standard scaling and covering argument easily  yields the claim in the original setting of half-balls.

\medskip

It remains to prove that $u$ satisfies the transmission condition in the classical sense.
Let $x_0 \in T_{3/4}$.
By piecewise $C^{1,\alpha}$ regularity, there is a piecewise linear function $\ell$ such that
\[
|u - \ell| \leq C_0 r^{1+\alpha} \quad \text{ in } B_r(x_0) \subset B_1,
\]
for all $0< r \leq 1/4$, for some constant $C_0 >0$.
For such an $r> 0$ and $0 < \vep < 1$, let $w$ be the piecewise quadratic barrier given by
\[
w(x) := \ell(x) 
+ C_0 r^{\alpha-1} \left( \frac{1}{\vep} \left(2 r |x_n| -  x_n^2 \right)
+ |x'- x_0'|^2 \right).
\]
Choosing $\vep > 0$ sufficiently small, we ensure that
\begin{equation}
\label{eq:contra}
\textstyle
\mcal^{+} (D^2 w; \frac{\lambda}{n}, \Lambda) < F^{\pm}(0).
\end{equation}
Recalling that $x_0 \in T$, on the boundary $\partial B_r(x_0)$ we have that
\[
w - \ell 
\geq C_0 r^{\alpha-1} \left( 
\vep^{-1} 
r |x_n| + |x'-x_0'|^2\right)
\geq C_0 r^{\alpha-1} |x- x_0|^2 = C_0 r^{1 +\alpha} \geq u - \ell
\]
and hence
\[
w \geq u \quad \text{ on } \partial B_r(x_0).
\]
Since $w(x_0) = u(x_0)$, there is a vertical translation $w + c$, with $c \geq 0$, touching $u$ from above at some $x_r \in B_r(x_0)$.
In fact, $x_r \in T_r(x_0)$, since otherwise the bulk equations would violate~\eqref{eq:contra}.

\medskip

By the transmission condition at $x_r$, we conclude
\[
G(\nabla w^{+}(x_r), \theta \nabla w^{-}(x_r)) \geq g(x_r),
\]
or equivalently
\[
G(\nabla \ell^{+} + 2 C_0 r^{\alpha-1}[ \vep^{-1} r e_n + x_r'-x_0'], \theta 
\nabla \ell^{-}
+ \theta 2 C_0 r^{\alpha-1}[ - \vep^{-1} r e_n + x_r'-x_0']
) \geq g(x_r),
\]
whence by uniform ellipticity
\[
g(x_r) \leq G(\nabla \ell^{+}, \theta \nabla \ell^{-}) + 4C_0 \Lambda r^{\alpha} \vep^{-1} + 4 C_0 \mu r^{\alpha},
\]
and letting $r \downarrow 0$
\[
g(x_0) \leq G(\nabla \ell^{+}, \theta \nabla \ell^{-}) = G(\nabla u^{+}(x_0), \theta \nabla u^{-}(x_0)).
\]

\medskip

Using $2 \ell- w$ instead of $w$ and arguing similarly, the other inequality follows.
\end{proof}

\medskip

For linear equations, we can keep iterating Theorem~\ref{thm:piecewise} to deduce piecewise smoothness:

\begin{cor}
Let $A^{\pm} \in \ellmat$, $\gamma^{\pm} \in \ellvec$, $f^{\pm} \in \R$, $g \in \R$, and $0 \leq \theta \leq 1$ be constants.
Let $u \in C(B_1)$ be a viscosity solution to the linear constant-coefficient transmission problem
\[
\begin{cases}
\tr{A^{\pm} D^2 u} = f^{\pm} & \text{ in } B_1^{\pm},\\
\gamma^{+} \cdot \nabla u^{+}  - \theta \gamma^{-} \cdot \nabla u^{-} = g
 & \text{ on } T_1.
\end{cases}
\]
Then $u$ is piecewise $C^{\infty}$ with an estimate
\[
\|u\|_{C^k(\overline{B}_{1/4})} \leq C_k \left(\|u\|_{L^{\infty}(B_1)} + |f^{+}| +|f^{-}| + |g| \right)
\]
for each $k \geq 0$, where $C_k$ depends only on $n$, $\lambda$, $\Lambda$, $\mu$, and $k$.
\end{cor}
\begin{proof}
By Theorem~\ref{thm:piecewise}, we know that $u^{\pm} \in C^{1,\overline{\alpha}}(B_1^{\pm} \cup T_1)$
satisfies an estimate
\begin{equation}
\label{eq:wha}
\|u^{\pm}\|_{C^{1,\overline{\alpha}}(\overline{B}^{\pm}_{1/2})} \leq C \left(\|u\|_{L^{\infty}(B_1)} + |f^{+}| +|f^{-}| + |g| \right) =: C K,
\end{equation}
where $K := \|u\|_{L^{\infty}(B_1)} + |f^{+}| +|f^{-}| + |g|$.

\medskip

Differentiating the equations in the horizontal direction, we see that the tangential derivatives of $u$ (which are continuous functions) satisfy the homogeneous transmission problem
\begin{equation}
\label{eq:invariant}
\begin{cases}
\tr{A^{\pm} D^2 v} = 0 & \text{ in } B_1^{\pm},\\
\gamma^{+} \cdot \nabla v^{+}  - \theta \gamma^{-} \cdot \nabla v^{-}= 0 & \text{ on } T_1,
\end{cases}
\end{equation}
in the viscosity sense.
By Theorem~\ref{thm:piecewise} (rescaled etc.) and~\eqref{eq:wha}, we deduce
\[
\|u_{x_i}\|_{C^{1,\overline{\alpha}}(\overline{B}_{r_2})} \leq C(r_1,r_2) \|u_{x_i}\|_{L^{\infty}(B_{r_1})} \leq C(r_1, r_2) K
\]
for $r_2 < r_1 \leq 1/2$, where $C(r_1,r_2) > 0$ depends only on $n$, $\lambda$, $\Lambda$, $\mu$, $r_1$, and $r_2$.
It follows in particular that $u|_{T} \in C^{2,\alpha}(\overline{T}_{r_2})$,
and by boundary regularity of the Dirichlet problem,
for $r_3 < r_2$ we have that
\[
\|u^{\pm}\|_{C^{2,\alpha}(\overline{B}_{r_3})} \leq C(r_2,r_3) \left( \|u|_{T}\|_{C^{2,\alpha}(\overline{T}_{r_2})} + \|u\|_{L^{\infty}(B_{r_2})}\right) \leq C(r_1, r_2, r_3) K,
\]
proving the claim for $k =2$ after an appropriate choice of $r_1$, $r_2$, $r_3$.

\medskip

Exploiting the invariance of~\eqref{eq:invariant} under horizontal translations, 
taking further tangential derivatives and arguing as above, it is not difficult to deduce the claim.
\end{proof}

%%%%%%%%%%%%%%%%%%%%%%%%%%%%%%%%%%%%%%%%%

\section{Equicontinuity up to the boundary}
\label{sec:barrier}

This section is devoted to the boundary behavior of viscosity solutions to transmission problems.
Our main result states that functions in the solution class $S$ are continuous up to the boundary, with a modulus of continuity depending only on the boundary datum and the source terms.
This is achieved by constructing appropriate barriers, combined with the ABP maximum principle.
Such an equicontinuity result is fundamental to the perturbation theories developed in Sections~\ref{sec:inhom} and~\ref{sec:pert}. 

\medskip

Below, by a \emph{modulus of continuity} $\omega$ for a function $u$ in some set $\Omega \subset \R^n$, we mean a function $\omega \colon [0, \infty) \to [0, \infty)$ such that
\[
|u(x) - u(y)| \leq \omega(|x-y|) \qquad \text{ for all } x \text{ and } y \text{ in } \Omega.
\]

\begin{thm}
[Equicontinuity up to the boundary-- c.f. Proposition~4.14 in~\cite{CC}]
\label{thm:barriers}
Let $u \in C(\overline{B}_{1})$ satisfy $u \in S(\lambda, \Lambda, \mu, \theta; f^{+}, f^{-}, g)$ in $B_1$, where $f^{\pm} \in C(B_1^{\pm} \cup T_1)$ and $g \in C(T_1)$.
Let $\phi := u|_{\partial B_1}$ and let $\omega(|x-y|)$ be a modulus of continuity for $\phi$ on $\partial B_1$.
Assume that $K$ is a positive constant such that
\[
\|\phi\|_{L^{\infty}(\partial B_1)} + \|f^{+}\|_{L^{n}(B_1^{+})} + \|f^{-}\|_{L^{n}(B_1^{-})} + \|g\|_{L^{\infty}(T_1)} \leq K.
\]
Then there is a modulus of continuity $\omega^{\star}$ of $u$ in $\overline{B}_{1}$, depending only on $n$, $\lambda$, $\Lambda$, $\mu$, $K$, and $\omega$.
\end{thm}

\begin{proof}
By interior estimates, it suffices to construct a modulus of continuity near the boundary. 
Since the behavior close to the interface differs essentially from that away from it, we distinguish two types of boundary moduli:
\begin{itemize}
\item for each $\kappa > 0$, a modulus $\omega_\kappa$ such that
\[
|u(x) - u(x_0)| \leq \omega_{\kappa}(|x-x_0|) \quad \text{ for all } x \in B_1 \text{ and } x_0 \in 
\{|x_n| \geq \kappa\} \cap \partial B_1;
\]
\item a modulus $\omega_0$ such that
\[
|u(x) - u(x_0)| \leq \omega_0(|x-x_0|) \quad \text{ for all } x \in B_1 \text{ and } x_0 \in \{x_n = 0\} \cap \partial B_1.
\]
\end{itemize}
Once the two classes of moduli are available, it is not difficult to combine them with the interior $C^{\alpha}$ estimates (Section~\ref{sec:holder}) into a modulus $\omega^{\star}$ in the full ball $\overline{B}_1$.
The interested reader may find the technical details in Appendix~\ref{app:boundary}.

\medskip

The construction of the first family $\{\omega_{\kappa}\}_{\kappa > 0}$ can be reduced to the classical theory.
Indeed, for each $\kappa > 0$, we can find a neighborhood $B_{\delta_{\kappa}}(x_0) \cap B_1$ of $x_0 \in \{|x_n| \geq \kappa\} \cap \partial B_1$, with $\delta_{\kappa}> 0$ depending only on $\kappa$, which does not intersect the interface $\{x_n = 0\}$.
Following the proof of Proposition~4.14 in~\cite{CC}, using the same linear barrier $h(x) = 1 - x \cdot x_0$ and the classical ABP, 
we obtain a modulus of continuity $\omega_{\kappa}$ as above.

\medskip

Therefore, it remains to construct the interface modulus $\omega_0$.
For simplicity, throughout the proof, we write $T$ for the ``global'' interface
\[
T = \{x_n = 0\}.
\]

\medskip

Let $\vep >0$.
We must show that 
\begin{equation}
\label{anima}
|u(x) - u(x_0)| \leq \vep \quad \text{ for all } x \in B_1 \text{ and } x_0 \in T \cap \partial B_1
\text{ such that } |x-x_0| \leq \delta,
\end{equation}
for some $\delta > 0$ depending only on $\vep$, $n$, $\lambda$, $\Lambda$, $\mu$, $K$, and $\omega$.

\medskip

We show~\eqref{anima} via a piecewise smooth barrier constructed in our previous paper~\cite{E}:

\begin{lem}[see the proof of Lemma~5.2 in~\cite{E}]
For each $x_0 \in \partial B_1 \cap T$, there is a function $h \in C(\overline{B}_1)$, with $h^{\pm} \in C^{\infty}(B_1^{\pm} \cup T_1)$,
such that
\begin{equation}
\label{thebar}
\begin{cases}
\mcal^{+}(D^2 h; \frac{\lambda}{n}, \Lambda) < 0 & \text{ in } B_1^{\pm},\\
\tcal^{+}(\nabla h^{+}, \theta \nabla h^{-}; \lambda, \Lambda, \mu) < 0 & \text{ on } T_1,\\
\end{cases}
\end{equation}
\begin{equation}
\label{posito}
h > 0 \quad \text{ on } \overline{B}_1 \setminus \{x_0\},
\end{equation}
\begin{equation}
\label{barthe}
h(x) \leq C |x-x_0|^{\alpha} \quad \text{ on } \overline{B}_1,
\end{equation}
where $0 < \alpha < 1$ and $C > 0$ are universal constants.
\end{lem}

\medskip

Let $x_0 \in \partial B_1$ and consider $h$ as above.
Take $\delta_1 > 0$ depending on $\vep$ and $\omega$ such that
\[
|\varphi(x) - \varphi(x_0)| = |u(x)- u(x_0)| \leq \vep \quad \text{ for all } x \in \partial B_1 \cap B_{\delta_1}(x_0).
\]
By the ABP estimate
\begin{equation}
\label{cos3}
|u(x) - u(x_0) \pm \vep| \leq 2 \|u\|_{L^{\infty}(B_1)} + \vep \leq C_1 \quad \text{ for all } x \in B_1,
\end{equation}
where $C_1 > 0$ depends only on $\vep$, $n$, $\lambda$, $\Lambda$, $\mu$, and $K$.
For $0 < \delta_2 \leq \delta_1$, define
\begin{equation}
\label{cos4}
h_{\delta_2} := \inf_{B_1 \cap \partial B_{\delta_2}(x_0)} h,
\end{equation}
and note that $h_{\delta_2} > 0$ by~\eqref{posito}.

\medskip

Let $\Omega := B_1 \cap B_{\delta_2}(x_0)$ and define the barrier functions
\[
\underline{\varphi}(x) = u(x)- u(x_0) - \vep - C_1 \frac{h(x)}{h_{\delta_2}}
\qquad \text{ and } \qquad
\overline{\varphi}(x) = u(x)- u(x_0) + \vep + C_1 \frac{h(x)}{h_{\delta_2}}.
\]
By~\eqref{thebar},~\eqref{cos3}, and~\eqref{cos4}, the functions $\underline{\varphi}$ and $\overline{\varphi}$ satisfy
\[
\begin{cases}
\underline{\varphi} \in \underline{S}(\lambda, \Lambda, \mu, \theta; f^{+}, f^{-}, g) & \text{ in } \Omega,\\
\underline{\varphi} \leq 0 & \text{ on } \partial \Omega
\end{cases}
\quad \text{ and } \quad
\begin{cases}
\overline{\varphi} \in \overline{S}(\lambda, \Lambda, \mu, \theta; f^{+}, f^{-}, g) & \text{ in } \Omega,\\
\overline{\varphi} \geq 0 & \text{ on } \partial \Omega.
\end{cases}
\]
The ABP maximum principle in $\Omega$ then gives that
\[
|u(x) - u(x_0)| \leq \vep + C_2 \diam{\Omega} + C_1\frac{h(x)}{h_{\delta_2}} \leq \vep + C_2 \delta_2 + C_1\frac{h(x)}{h_{\delta_2}} \quad \text{ in } \Omega,
\]
for some $C_2$ depending only on $n$, $\lambda$, $\Lambda$, $\mu$, and $K$.
Taking $\delta_2$ small such that $C_2 \delta_2 \leq \vep$, we obtain
\[
|u(x) - u(x_0)| \leq 2 \vep + C_1\frac{h(x)}{h_{\delta_2}} \quad \text{ in } \Omega.
\]
Using~\eqref{barthe} and taking $\delta \leq \delta_2$ sufficiently small, we finally deduce
\[
|u(x) - u(x_0)| \leq 3 \vep \quad \text{ in } B_1 \cap B_{\delta}(x_0),
\]
which is equivalent to~\eqref{anima}.
\end{proof}

Combined with our comparison principle (Theorem~\ref{thm:comp}),
Theorem~\ref{thm:barriers} allows us to approximate viscosity solutions to constant-coefficient inhomogeneous transmission problems by their homogeneous replacements.
This result is fundamental to the perturbation theory for such equations.

\begin{lem}
[Approximation I]
\label{lem:silly}
Let $u \in C(B_1)$ be a viscosity solution to 
\[
\begin{cases}
F^{\pm}(D^2 u) = f^{\pm}(x) & \text{ in } B_{1}^{\pm},\\
G(\nabla u^{+}, \theta \nabla u^{-}) = g(x) & \text{ on } T_{1},
\end{cases}
\]
and let $h \in C(\overline{B}_{3/4})$ be the unique viscosity solution to the boundary-value problem
\[
\begin{cases}
F^{\pm}(D^2 h) = 0 & \text{ in } B_{3/4}^{\pm},\\
G(\nabla h^{+}, \theta \nabla h^{-}) = 0 & \text{ on } T_{3/4},\\
h = u & \text{ on } \partial B_{3/4}.
\end{cases}
\]
Assume that $\|u\|_{L^{\infty}(B_1)} \leq 1$ and that
\[
\|f^{+}\|_{L^{n}(B^{+}_{1})}  + \|f^{-} \|_{L^{n}(B^{-}_{1})} + \|g \|_{L^{\infty}(T_{1})} \leq \delta.
\]

For each $\vep > 0$, there is $\delta >0$ small depending only on $\vep$, $n$, $\lambda$, $\Lambda$, and $\mu$ such that 
\[
\|u - h\|_{L^{\infty}(B_{1/2})} 
\leq \vep.
\]
\end{lem}
\begin{proof}
By the interior $C^{\alpha}$ estimate
\[
[u]_{C^{\alpha}(\overline{B}_{3/4})} \leq C (1 + \delta),
\]
hence $u$ has a H\"{o}lder modulus of continuity $\omega(r) := C(1+\delta) r^{\alpha}$ on $\overline{B}_{3/4}$,
with $\alpha$ and $C$ universal.
By Theorem~\ref{thm:barriers}, the function $h$ has a modulus $\omega_{\delta}^{\star}$ on $\overline{B}_{3/4}$
depending only on $n$, $\lambda$, $\Lambda$, $\mu$, and $\delta$.
Since $u - h = 0$ on $\partial B_{3/4}$, 
for $0 < \eta < 1/4$,
by the boundary moduli we have that
\[
\|u - h\|_{L^{\infty}(\partial B_{3/4 - \eta})}
\leq \omega(\eta) + \omega^{\star}_\delta(\eta).
\]
Moreover, by Theorem~\ref{thm:comp}, the difference $u-h$ belongs to the solution class $S(\lambda, \Lambda, \mu, \theta; f^{+}, f^{-}, g)$ in $B_{3/4}$, whence by the ABP maximum principle
\[
\|u -h\|_{L^{\infty}(B_{3/4-\eta})}
\leq \|u -h\|_{L^{\infty}(\partial B_{3/4-\eta})} + C \delta
\leq \omega(\eta)+ \omega_{\delta}^{\star}(\eta) + C \delta.
\]
Therefore, first choosing $\delta >0$ small such that $C \delta \leq \vep/2$ and then choosing $\eta > 0$ small such that $\omega(\eta) + \omega_{\delta}(\eta) \leq \vep/2$, the claim follows.
\end{proof}

\medskip

%%%%%%%%%%%%%%%%%%%%%%%%%%%%%%%%%%%%%%%%%

\section{Perturbation theory I: constant coefficients}
\label{sec:inhom}

In this section, we extend the piecewise $C^{1,\overline{\alpha}}$ regularity of the constant-coefficient homogeneous problem (Theorem~\ref{thm:piecewise}) to the inhomogeneous setting.
This auxiliary result is crucial for the regularity theory of variable-coefficient problems.

\medskip

Recall the definition of the $C^{\alpha-1}$-Morrey norm~\eqref{morrey} in the Introduction.
Our efforts are devoted to proving the following theorem:

\begin{thm}
[Piecewise $C^{1,\alpha}$ regularity for constant-coefficient inhomogeneous problems]
\label{thm:const}
Let $\{F^{\pm}, G\}$ be constant-coefficient operators with ellipticity constants $0 < \lambda \leq \Lambda$ and $\mu \geq 0$, and assume that $F^{\pm}(0) = G(0,0) = 0$.
Let $0 < \alpha < \overline{\alpha}$, where $\overline{\alpha}$ is the universal exponent in the conclusion of Theorem~\ref{thm:piecewise}.
Let $f^{\pm}$ be continuous functions on $B_1^{\pm} \cup T_1$, let $g$ be a $C^{\alpha}$ function in $T_1$, and let $\theta \in [0,1]$ be a constant.
Let $u \in C(B_1)$ be a viscosity solution to the transmission problem
\[
\begin{cases}
F^{\pm}(D^2 u) = f^{\pm}(x) & \text{ in } B_{1}^{\pm},\\
G(\nabla u^{+}, \theta \nabla u^{-}) = g(x) &\text{ on } T_{1}.
\end{cases}
\]

Then $u^{\pm}$ are $C^{1,\alpha}$ on $B_1^{\pm} \cup T_1$, and moreover
\[
\|u^{\pm}\|_{C^{1,\alpha}(\overline{B}_{1/2}^{\pm})} \leq C \left( \|u\|_{L^{\infty}(B_1)} 
+ \|f^{+}\|_{C^{\alpha-1}(B_1^{+})} 
+ \|f^{-}\|_{C^{\alpha-1}(B_1^{-})} 
+ \|g\|_{C^{\alpha}(\overline{T}_1)}
\right)
\]
for some constant $C > 0$ depending only on $n$, $\lambda$, $\Lambda$, $\mu$, $\alpha$, and $\overline{\alpha}$.
\end{thm}

\begin{rem}
Here, notice that there is no loss in assuming that $F^{\pm}(0) = 0$ and $G(0,0) = 0$, since otherwise we may replace the functions $f^{\pm}$ and $g$ by $f^{\pm} - F^{\pm}(0)$ and $g - G(0,0)$, respectively.
\end{rem}

\medskip

The $C^{1,\alpha}$ estimate away from the interface is classical (see Section 8.2 in~\cite{CC}), hence by a standard scaling and covering argument it suffices to prove a piecewise $C^{1,\alpha}$ bound at a single point on the interface, namely, the origin.
To this aim, we make several simplifying assumptions.

\medskip

It is convenient to give the pointwise estimate in terms of more accurate \emph{punctual} norms of the sources.
We define the punctual $C^{\alpha-1}$ norm of $f^{\pm} \in C(B_R^{\pm} \cup T_R)$ at $x \in T_R$ by
\begin{equation}
\label{point:morrey}
\|f^{\pm}\|_{C_{[x]}^{\alpha-1}(B^{\pm}_R)} := \sup_{B_r(x) \subset B_R} r^{-\alpha} \|f^{\pm}\|_{L^{n}(B^{\pm}_r(x))}.
\end{equation}
Similarly, for $g \in C(T_R)$ and $x \in T_R$, we consider the punctual H\"{o}lder seminorm
\begin{equation}
\label{point:holder}
[g]_{C^{\alpha}_{[x]}(T_R)} := \sup_{y \in T_R, \, y \neq x} \frac{|g(x) - g(y)|}{|x-y|^{\alpha}}.
\end{equation}

\medskip

By uniform ellipticity of $G$, recalling that $G(0, 0) = 0$, we can find a unique $t \in \R$ such that $G(t e_n, 0) = g(0)$ and satisfying the estimate
\[
|t| \leq \frac{1}{\lambda} |g(0)|.
\]
Therefore, replacing the functions
\[
u(x), \qquad\qquad G(\xi^{+}, \xi^{-}), \qquad \qquad g(x),
\]
by
\[
u(x)- t x_n^{+}, \qquad G(\xi^{+} + t e_n, \xi^{-}) - g(0), \qquad g(x) - g(0),
\]
respectively, we may additionally assume that $g(0) = 0$.

\medskip

Let $\delta > 0$.
Rescaling $u$ by 
\[
\widetilde{u}(x) = \frac{u(x)}{\|u\|_{L^{\infty}(B_1)} 
+ \delta^{-1} 
\Big( 
\|f^{+}\|_{C_{[0]}^{\alpha-1}(B_1^{+})} + \|f^{-}\|_{C_{[0]}^{\alpha-1}(B_1^{-})} 
+ [g]_{C_{[0]}^{\alpha}(T_1)}
\Big) 
} =: \frac{u(x)}{K},
\]
replacing the operators $\{F^{\pm}, G\}$ by
\[
\widetilde{F}^{\pm}(M) = \frac{1}{K} F(K M), \qquad \widetilde{G}(\xi^{+}, \xi^{-}) = \frac{1}{K} G(K \xi^{+}, K \xi^{-}),
\]
and the source terms by
\[
\widetilde{f}^{\pm}(x) = \frac{1}{K} f^{\pm}(x), \qquad \widetilde{g}(x) = \frac{1}{K} g(x),
\]
we may further assume that
$\|u\|_{L^{\infty}(B_1)} \leq 1$ and
\[
\|f^{\pm}\|_{L^{n}(B_r^{\pm})} \leq \delta r^{\alpha},
\qquad 
\|g\|_{L^{\infty}(T_{r})} \leq \delta r^{\alpha}
\qquad \text{ for } r \leq 1.
\]
Here, note that $\{\widetilde{F}^{\pm}, \widetilde{G}\}$ have the same ellipticity constants as $\{F^{\pm}, G\}$.

\medskip

We state the following variant of the piecewise $C^{1,\overline{\alpha}}$ estimate (that can be easily obtained from Theorem~\ref{thm:piecewise} by a scaling and covering argument) for reference below:
\begin{itemize}
\item
If $w \in C(B_{3/4})$ is a viscosity solution to the constant-coefficient transmission problem
\[
\begin{cases}
F^{\pm}(D^2 w) = 0 & \text{ in } B_{3/4}^{\pm},\\
G(\nabla w^{+}, \theta \nabla w^{-}) = 0 & \text{ on } T_{3/4},
\end{cases}
\]
with $F^{\pm}(0) = G(0, 0) = 0$,
then $w$ is piecewise $C^{1,\overline{\alpha}}$ with an estimate
\begin{equation}
\label{fund2}
\|w^{\pm}\|_{C^{1,\overline{\alpha}}(\overline{B}^{\pm}_{1/2})} 
\leq C_0 \|w\|_{L^{\infty}(B_{3/4})},
\end{equation}
where $C_0 > 0$ is a universal constant.
\end{itemize}

\medskip

From the discussion above, it remains to show the following:

\begin{thm}
Let $\{F^{\pm}, G\}$ be constant-coefficient operators with ellipticity constants $0 < \lambda \leq \Lambda$ and $\mu \geq 0$, let $f^{\pm} \in C(B_1^{\pm} \cup T_1)$, let $g \in C(T_1)$, and let $\theta \in [0,1]$.
Let $u \in C(B_1)$ be a viscosity solution to 
\[
\begin{cases}
F^{\pm}(D^2 u) = f^{\pm}(x) & \text{ in } B_{1}^{\pm},\\
G(\nabla u^{+}, \theta \nabla u^{-}) = g(x) &\text{ on } T_{1}.
\end{cases}
\]
Assume that 
\begin{equation}
\label{ao1}
F^{\pm}(0) = 0, \qquad G(0,0) = 0, \qquad g(0) = 0, \qquad \text{and} \qquad \|u\|_{L^{\infty}(B_1)} \leq 1.
\end{equation}

For each $0 < \alpha < \overline{\alpha}$, there exists $\delta > 0$ depending only on $n$, $\lambda$, $\Lambda$, $\mu$, $\alpha$, and $\overline{\alpha}$ such that if
\begin{equation}
\label{ao2}
\|f^{\pm}\|_{L^n(B_r^{\pm})} \leq \delta r^{\alpha}
\qquad \text{ and }
\qquad \|g\|_{L^{\infty}(T_r)} \leq \delta r^{\alpha}
\qquad \text{ for all } r \leq 1,
\end{equation}
then $u$ is piecewise $C^{1,\alpha}$ at the origin, i.e., there is a piecewise linear function
\[
\ell(x) = a^{+} \cdot x \, \chara{\{x_n \geq 0\}} + a^{-} \cdot x \, \chara{\{x_n < 0\}} + u(0),
\]
with $a^{\pm} \in \R^n$, such that
\[
|a^{\pm}| \leq C,
\]
\[
G(a^{+}, \theta a^{-}) = 0,
\]
\[
\|u - \ell\|_{L^{\infty}(B_r)} \leq 
Cr^{1 +\alpha}
\quad \text{ for all } r \leq 1,
\]
for some constant $C > 0$ depending only on $n$, $\lambda$, $\Lambda$, $\mu$, $\alpha$, and $\overline{\alpha}$.
\end{thm}

\begin{proof}
We will construct a sequence of piecewise linear functions
\begin{equation}
\label{con0}
\ell_k(x) = a_k^{+} \cdot x \, \chara{\{x_n \geq 0\}} + a_k^{-} \cdot x \, \chara{\{x_n < 0\}} + c_k,
\end{equation}
with $a_k^{\pm} \in \R^n$ and $c_k \in \R$ such that, for all $k \geq 0$
\begin{equation}
\label{con1}
G(a_k^{+}, \theta a_k^{-}) = 0,
\end{equation}
\begin{equation}
\label{con2}
\|u - \ell_k\|_{L^{\infty}(B_{\rho^{k}})} \leq \rho^{k(1+\alpha)},
\end{equation}
\begin{equation}
\label{con3}
|a_k^{\pm} - a_{k-1}^{\pm}| \leq C_0 \rho^{(k-1)\alpha}, \qquad |c_k - c_{k-1}| \leq C_0 \rho^{(k-1)(1+\alpha)},
\end{equation}
where $0 < \rho < 1$ is a constant depending only on $n$, $\lambda$, $\Lambda$, $\mu$, $\alpha$, and $\overline{\alpha}$.
It is then standard to verify that $\ell = \lim_{k \to \infty} \ell_k$ satisfies the properties in the statement.

\medskip

Let $\ell_{-1} = \ell_0 = 0$ and note that the above equations hold for $k = 0$.
Indeed,~\eqref{con1} and~\eqref{con2} follow from~\eqref{ao1}, while~\eqref{con3} is obvious.

\medskip

Assume by induction that we have constructed $\ell_k$ as above satisfying~\eqref{con1},~\eqref{con2}, and~\eqref{con3}, for some $k \geq 0$.
Let $v$ be the function
\[
v(x) := \frac{(u - \ell_k)(\rho^{k} x)}{\rho^{k(1+\alpha)}}
\]
and observe that it is a viscosity solution to the transmission problem
\begin{equation}
\label{apro1}
\begin{cases}
F_k^{\pm}(D^2 v) = f^{\pm}_k(x) & \text{ in } B_1^{\pm},\\
G_k(\nabla v^{+},\theta \nabla v^{-})
= g_k(x) & \text{ on } T_1,
\end{cases}
\end{equation}
with $F_k^{\pm}$, $f_k^{\pm}$, $G_k$, and $g_k$ given by
\[
F_k^{\pm}(M) := \rho^{k (1-\alpha)} F(\rho^{-k (1-\alpha)}M),
\qquad \qquad 
f^{\pm}_k(x) := \rho^{k (1-\alpha)} f^{\pm}(\rho^{k} x),
\]
\[
G_k(\xi^{+}, \xi^{-}) := \rho^{-k \alpha}G(a_k^{+} + \rho^{k\alpha} \xi^{+}, \theta a_k^{-} + \rho^{k\alpha} \xi^{-}),
\qquad \qquad
g_k(x) := \rho^{-k \alpha} g(\rho^{k} x).
\]

\medskip

The operators $\{F_{k}^{\pm}, G_k\}$ have ellipticity constants $0 < \lambda \leq \Lambda$ and $\mu \geq 0$.
By~\eqref{ao1} and the induction hypotheses~\eqref{con1} and~\eqref{con2}, we have that
\begin{equation}
\label{oa1}
F_{k}^{\pm}(0) = 0, \qquad G_{k}(0,0) = 0, \qquad g_{k}(0) = 0, \qquad 
\text{and} \qquad 
\|v\|_{L^{\infty}(B_1)} \leq 1.
\end{equation}
Moreover, by~\eqref{ao2} the source terms are small, with
\begin{equation}
\label{oa2}
\|f_k^{\pm}\|_{L^{n}(B_1^{\pm})} \leq \delta \qquad \text{ and } \qquad \|g_k\|_{L^{\infty}(T_1)} \leq \delta.
\end{equation}

\medskip

Let $w$ be the viscosity solution (whose existence and uniqueness is guaranteed by Theorem~\ref{thm:wellposed}) to the transmission problem
\begin{equation}
\label{apro2}
\begin{cases}
F^{\pm}_k(D^2 w ) = 0 & \text{ in } B_{3/4}^{\pm},\\
G_k(\nabla w^{+}, \theta \nabla w^{-}) = 0 & \text{ on } T_{3/4},\\
w = v & \text{ on } \partial B_{3/4}.
\end{cases}
\end{equation}
By~\eqref{fund2} and~\eqref{oa1}, the function $w$ is piecewise $C^{1,\overline{\alpha}}$ and satisfies the estimate
\begin{equation}
\label{hest}
\|w^{\pm}\|_{C^{1,\overline{\alpha}}(\overline{B}^{\pm}_{1/2})} 
\leq C_0 \|w\|_{L^{\infty}(B_{3/4})}
\leq C_0 \|v\|_{L^{\infty}(\partial B_{3/4})} \leq C_0,
\end{equation}
where we have used the ABP estimate for~\eqref{apro2} and that $\|v\|_{L^{\infty}(\partial B_{3/4})} \leq 1$.
In particular, $w$ has a piecewise linear Taylor approximation $\overline{\ell}$ at $0$ given on each side by
\[
\overline{\ell}^{\pm}(x) := 
\nabla w^{\pm}(0) \cdot x + w(0)
\]
and from~\eqref{hest} we see that
\begin{equation}
\label{hgood}
\|w - \overline{\ell}\|_{L^{\infty}(B_{\rho})} \leq C_0 \rho^{1+\overline{\alpha}} \leq \frac{1}{2} \rho^{1+\alpha}
\end{equation}
by choosing $0 < \rho < 1/2$ small depending only on $C_0$ (universal) and $\alpha < \overline{\alpha}$.
Applying Lemma~\ref{lem:silly} to~\eqref{apro1} and~\eqref{apro2},
we deduce the existence of a constant
$\delta > 0$ 
depending only on $n$, $\lambda$, $\Lambda$, $\mu$, $\alpha$, and $\overline{\alpha}$ such that when taken in~\eqref{oa2} yields
\begin{equation}
\label{vgood}
\|v - w\|_{L^{\infty}(B_{1/2})} \leq \frac{1}{2}\rho^{1+\alpha}.
\end{equation}
Combining~\eqref{hgood} and~\eqref{vgood}, we conclude
\begin{equation}
\label{algo}
\|v - \overline{\ell}\|_{L^{\infty}(B_{\rho})} \leq \|v - w\|_{L^{\infty}(B_{\rho})} + \|w - \overline{\ell}\|_{L^{\infty}(B_{\rho})} \leq \rho^{1+\alpha}.
\end{equation}

\medskip

We define the $(k+1)$-th polynomial in~\eqref{con0} by
\[
\ell_{k+1}(x) := \ell_k(x) + \rho^{k(1+\alpha)} \overline{\ell}(\rho^{-k} x)
\]
or equivalently, with coefficients
\[
a_{k+1}^{\pm} := a_k^{\pm} + \rho^{k\alpha} \nabla w^{\pm}(0) \qquad \text{ and } \qquad c_{k+1} := c_k + \rho^{k(1+\alpha)} w(0).
\]
To conclude the proof, we must check that~\eqref{con1},~\eqref{con2}, and~\eqref{con3} hold for $k+1$.

\medskip

Since $w$ satisfies the transmission condition from~\eqref{apro2} in the classical sense, we have that
\[
G_k(\nabla w_{k}^{+}(0), \theta \nabla w_{k}^{-}(0)) = 0,
\]
which is equivalent to~\eqref{con1} by definition of $G_k$ and $a_{k+1}^{\pm}$.
From~\eqref{algo} we deduce~\eqref{con2} as
\[
\|u - \ell_{k+1}\|_{L^{\infty}(B_{\rho^{k+1}})} = \rho^{k(1+\alpha)}\|v - \overline{\ell}\|_{L^{\infty}(B_{\rho})} \leq \rho^{(k+1)(1+\alpha)}.
\]
Finally, from~\eqref{hest} we see that
\[
|a_{k+1}^{\pm} - a_{k}^{\pm}| = \rho^{k \alpha} |\nabla w^{\pm}(0)| \leq C_0 \rho^{k\alpha}
\qquad \text{ and } \qquad |c_{k+1} - c_{k}| = \rho^{k (1+\alpha)} |w(0)| \leq C_0 \rho^{k(1+\alpha)},
\]
whence~\eqref{con3} follows.
\end{proof}

\medskip

%%%%%%%%%%%%%%%%%%%%%%%%%%%%%%%%%%%%%%%%%

\section{Perturbation theory II: variable coefficients}
\label{sec:pert}

In this section we prove Theorem~\ref{thm:main}, the piecewise $C^{1,\alpha}$ regularity of solutions to variable-coefficient transmission problems.
We do this by perturbation from constant-coefficient equations, for which optimal regularity is already known.
As explained in the Introduction, for this argument to work, we need a stronger notion of viscosity solution.

\medskip

First, we introduce a more general class of test functions whose Hessians from either side may blow up at the interface:

\begin{defn}
[Admissible* function]
The function $\varphi \in C(\Omega)$ is 
\emph{admissible*}
if its restrictions $\varphi^{\pm}$ are $C^2$ in the interior and $C^{1}$ up to the interface, i.e., if
\[
\varphi^{\pm} \in C^2(\Omega^{\pm}) \cap C^{1}(\Omega^{\pm} \cup T).
\]
\end{defn}

This class yields the following more restrictive notion of viscosity solution:

\begin{defn}[Viscosity* solution]
\label{def:viscstar}
We say that $u \in C(\Omega)$ is a \emph{viscosity*} \emph{subsolution} to~\eqref{eq:nl} if whenever an admissible* function $\varphi$ touches $u$ from above at $x_0 \in \Omega$ and:
\begin{itemize}
\item $x_0 \in \Omega^{\pm}$, then
\[
F^{\pm}(D^2 \varphi(x_0), x_0) \geq f^{\pm}(x_0);
\]
\item $x_0 \in T$, then
\[
G(\nabla \varphi^{+}(x_0), \theta \nabla \varphi^{-}(x_0), x_0) \geq g(x_0).
\]
\end{itemize}
We have an analogous definition of viscosity* supersolution.
A viscosity* solution is a function that is both a viscosity* subsolution and supersolution.
\end{defn}

\begin{rem}
In~\cite{E}, we showed that standard viscosity solutions in the sense of Definition~\ref{def:visc} can be equivalently defined with respect to \emph{admissible} test functions $\varphi \in C(\Omega)$ satisfying
\[
\varphi^{\pm} \in C^{2}(\Omega^{\pm}) \cap C^{1,1}(\Omega^{\pm} \cup T),
\]
i.e., whose second derivatives remain bounded up to the interface.
Since the admissible* class is larger than the admissible class, it follows that every viscosity* solution is a viscosity solution.
Our perturbation theory for variable-coefficient equations relies on the use of a non-explicit barrier that we only know to be admissible*, hence our need for a stronger notion of viscosity solution.
\end{rem}

\begin{rem}
We note that sufficiently smooth viscosity solutions are also viscosity* solutions. 
More precisely, if $u \in C(\Omega)$, with $u^{\pm} \in C^{1}(\Omega^{\pm} \cup T)$, is a viscosity solution, then it is also a viscosity* solution.
\end{rem}

\medskip

Let $u \in C(B_1)$ be a viscosity* solution to 
\begin{equation}
\label{calarca}
\begin{cases}
F^{\pm}(D^2 u, x) = f^{\pm}(x) & \text{ in } B_1^{\pm},\\
G(\nabla u^{+}, \theta \nabla u^{-}, x) = g(x) & \text{ on } T_1,
\end{cases}
\end{equation}
with $F^{\pm}$, $G$, $f^{\pm}$, $g$, and $0 \leq \theta \leq 1$ as in the statement of Theorem~\ref{thm:main}.

\medskip

As in Section~\ref{sec:inhom}, the $C^{1,\alpha}$ estimate away from the interface is classical, hence by a standard scaling and covering argument it suffices to prove a pointwise bound at the origin.
Once this is established, the fact that the transmission condition at that point is satisfied in the classical sense follows from the same reasoning in Section~\ref{sec:barrier}.

\medskip

To quantify the regularity of the operators with respect to the independent variable $x$, we define the moduli of continuity at $0$ for the operators $F^{\pm}$ and $G$ by
\[
\beta^{\pm}(x) = \beta_{F^{\pm}}(x) := \sup_{M \in \scal_{n}} \frac{|F^{\pm}(M, x) - F^{\pm}(M,0)|}{\|M\| + 1}
\qquad \text{ for } x \in B_1^{\pm},
\]
and
\[
\gamma(x) = \gamma_{G}(x) 
:= \sup_{\xi^{\pm} \in \R^{n}} \frac{|G(\xi^{+}, \xi^{-}, x) - G(\xi^{+}, \xi^{-}, 0)|}{|\xi^{+}| + |\xi^{-}| + 1}
\qquad \text{ for } x \in T_1,
\]
respectively.

\medskip

Our perturbation is based on an approximation lemma allowing us to compare solutions to variable-coefficients equations with their constant-coefficient replacements.
More precisely, we have the following:

\begin{lem}
[Approximation II]
\label{lem:approx}
Let $\{F^{\pm}, G\}$ be operators with ellipticity constants $0 < \lambda \leq \Lambda$ and $\mu \geq 0$.
Let $f^{\pm}$ and $g$ be H\"{o}lder continuous functions on $B_1^{\pm} \cup T_1$ and $T_1$, respectively.
Let $\phi$ be a continuous function on $\partial B_1$, with modulus of continuity $\omega = \omega(r)$.
Let $0 \leq \theta \leq 1$ be a constant.

Assume that $\beta^{\pm}$ and $\gamma$ are H\"{o}lder continuous in $B_1^{\pm} \cup T_1$ and $T_1$, respectively, and that
\[
F^{\pm}(0,x) \, \text{ in } B_1^{\pm}, \qquad G(0,0,x) = 0 \, \text{ on } T_1, \qquad g(0) = 0,
\]
and
\[
\|\phi\|_{L^{\infty}(\partial B_1)} \leq K.
\]

For $\vep > 0$, there is $\delta >0$ depending only on $\vep$, $n$, $\lambda$, $\Lambda$, $\mu$, $\omega$ and $K$, such that if 
\[
\|\beta^{\pm}\|_{L^n(B^{\pm}_1)} \leq \delta, 
\qquad \|f^{\pm}\|_{L^n(B^{\pm}_1)} \leq \delta, 
\qquad \|\gamma\|_{L^{\infty}(T_1)} \leq \delta,
\qquad \text{ and }
\qquad \|g\|_{L^{\infty}(T_1)} \leq \delta, 
\]
then any viscosity* solution $v \in C(\overline{B}_1)$ to the variable-coefficient inhomogeneous problem
\[
\begin{cases}
F^{\pm}(D^2 v, x) = f^{\pm}(x) & \text{ in } B_1^{\pm},\\
G(\nabla v^{+}, \theta \nabla v^{-},x) = g(x) & \text{ on } T_1,\\ 
v = \phi & \text{ on } \partial B_1,
\end{cases}
\]
and any viscosity solution $w \in C(\overline{B}_1)$ to the constant-coefficient homogeneous problem
\[
\begin{cases}
F^{\pm}(D^2 w, 0) = 0 & \text{ in } B_1^{\pm},\\
G(\nabla w^{+}, \theta \nabla w^{-}, 0) = 0 & \text{ on } T_1,\\ 
w = \phi & \text{ on } \partial B_1,
\end{cases}
\]
satisfy
\[
\|v - w\|_{L^{\infty}(B_1)} \leq \vep.
\]
\end{lem}

\begin{rem}
Note here that $v$ is a viscosity* solution, while $w$ is a standard viscosity solution.
The meaning of the additional assumptions on the data will become evident later when we prove the main theorem.
\end{rem}

\begin{proof}
Suppose by contradiction that the statement of the lemma is false.
Then there are $\vep > 0$, $n$, $\lambda$, $\Lambda$, $\omega$, $K$, and sequences of operators $\{F_k^{\pm}, G_k\}$, functions $\{f_k^{\pm}, g_k, \phi_k\}$, and constants $\{\theta_k\} \subset [0, 1]$ as above, for which we can find viscosity* solutions $\{v_k\}$ and viscosity solutions $\{w_k\}$ to the problems
\[
\begin{cases}
F_k^{\pm}(D^2 v_k, x) = f_k(x) & \text{ in } B_1^{\pm}\\
G_k(\nabla v_k^{+}, \theta_k \nabla v_k^{-}, x) = g_k(x) & \text{ on } T_1\\
v_k = \phi_k & \text{ on } \partial B_1
\end{cases}
\quad \text{ and } \quad 
\begin{cases}
F_{k}^{\pm}(D^2 w_{k}, 0) = 0 & \text{ in } B_{1}^{\pm}\\
G_{k}(\nabla w_{k}^{+}, \theta_k \nabla w_{k}^{-}, 0) = 0
& \text{ on } T_{1}\\
w_{k} = \phi_k & \text{ on } \partial B_1,
\end{cases}
\]
respectively,
such that
\[
\|\beta_k^{+}\|_{L^n(B_1^{+})} + \|\beta_k^{-}\|_{L^n(B_1^{-})}
+ \|\gamma_k\|_{L^{\infty}(T_1)} + \|f_k^{+}\|_{L^{n}(B^{+}_1)} + \|f_k^{-}\|_{L^{n}(B^{-}_1)}+ \|g_k\|_{L^{\infty}(T_1)} \to 0,
\]
and
\[
\|v_k - w_k\|_{L^{\infty}(B_1)} > \vep \quad \text{ for all } k.
\]

\medskip

Note that the functions $\{v_k\}$ and $\{w_k\}$ have the same modulus of continuity $\omega$ on $\partial B_1$.
Since $\|\phi_k\|_{L^{\infty}(\partial B_1)} \leq K$ for all $k$, and $\|f_k^{\pm}\|_{L^n(B_1^{\pm})} \to 0$ and $\|g_k\|_{L^{\infty}(T_1)} \to 0$ as $k \to \infty$, by the boundary estimate (Theorem~\ref{thm:barriers}) we see that $\{v_k\}$ and $\{w_k\}$ are equicontinuous and uniformly bounded in~$\overline{B}_1$.
Thus, taking subsequences, we may assume that 
\[
v_k \to v_{\infty} \quad \text{ and } \quad w_k \to w_{\infty} \quad \text{ uniformly in } \overline{B}_{1}  \text{ as } k \to \infty.
\]
The limiting functions are continuous in $\overline{B}_{1}$, with
\[
v_{\infty} = w_{\infty} \quad \text{ on } \partial B_1.
\]
Since $\|\beta_k^{\pm}\|_{L^n(B_1^{\pm})} \to 0$ and  $\|\gamma_k\|_{L^{\infty}(T_1)} \to 0$, taking further subsequences we may assume that
\[
F_k^{\pm}(\cdot, 0) \to F^{\pm}(\cdot, 0) \qquad \text{ and } \qquad G_k(\cdot, \cdot, 0) \to G(\cdot, \cdot, 0)
\]
uniformly on compacts in $\scal_n$ and $\R^n \times \R^n$, respectively, as well as $\theta_k \to \theta$.

\medskip

To prove the lemma, it suffices to show that $v_{\infty}$ is a viscosity solution to the transmission problem
\begin{equation}
\label{guro}
\begin{cases}
F^{\pm}(D^2 v, 0) = 0 & \text{ in } B_{1}^{\pm},\\
G(\nabla v^{+}, \theta \nabla v^{-}, 0) = 0 & \text{ on } T_{1}.
\end{cases}
\end{equation}
Indeed, by the closedness property of viscosity solutions (see Proposition~2.8 in~\cite{E}), we already know that $w_{\infty}$ is a viscosity solution to~\eqref{guro} and hence by uniqueness $v_{\infty} = w_{\infty}$ in $\overline{B}_{1}$, contradicting $\|v_k - w_k\|_{L^{\infty}(B_1)}> \vep$ for large $k$.

\medskip

For simplicity, we only check the subsolution property.
To this aim, let $P$ be a piecewise quadratic function touching $v_{\infty}$ from above at $x_0 \in B_1$.
The case $x_0 \notin T_1$ is classical, since we can restrict the function to a small ball entirely contained in $B_1^{+}$ or $B_1^{-}$ and follow the proof of Lemma~8.2 in~\cite{CC}.
Hence, we assume that $x_0 \in T_1$.

\medskip

Suppose for the sake of contradiction that
\begin{equation}
\label{contra1}
G(\nabla P^{+}(x_0), \theta \nabla P^{-}(x_0), 0) = - \eta < 0.
\end{equation}
By Remark~\ref{rem:trick}, we may assume that $P$ touches $v_{\infty}$ strictly at $x_0$ and that
\begin{equation}
\label{contra2}
F^{\pm}(D^2 P^{\pm}, 0) < - 2\eta.
\end{equation}
Let $\vep_k \downarrow 0$ be such that 
\[
F^{\pm}_k(D^2 P^{\pm}, 0) - F^{\pm}(D^2 P^{\pm}, 0) \leq \vep_k,
\]
and
\[
G_k(\nabla P^{+}(x), \theta_k \nabla P^{-}(x), 0) - G(\nabla P^{+}(x), \theta \nabla P^{-}(x), 0) \leq \vep_k \quad \text{ for } x \in T_1.
\]

\medskip

Let $\psi_k \in C(\overline{B}_1)$ be the unique viscosity solution to the Dirichlet problem
\begin{equation}
\label{diri}
\begin{cases}
\mcal^{+}(D^2 \psi_k; \frac{\lambda}{n}, \Lambda) = 
- \|D^2 P^{\pm}\| \beta^{\pm}_k(x) - |f_k^{\pm}(x)| - \vep_k
& \text{ in } B_1^{\pm},\\
\tcal^{+}(\nabla \psi_k^{+}, \theta_k \nabla \psi_k^{-}; \lambda,\Lambda,\mu) = 
- \left\| \left(|\nabla P^{+}|+ |\nabla P^{-}|+1\right) \gamma_k \right\|_{L^{\infty}(T_1)}+ \|g_k\|_{L^{\infty}(T_1)} - \vep_k
& \text{ on } T_1,\\
\psi_k = 0 & \text{ on } \partial B_1,\\
\end{cases}
\end{equation}
whose existence is guaranteed by Theorem~\ref{thm:wellposed}.
Since~\eqref{diri} is a constant-coefficient transmission problem and $\mcal^{+}$ is a convex operator, by the Evans-Krylov $C^{2,\alpha}$ estimates combined with Theorem~\ref{thm:const}, it follows that $\psi_k$ is and admissible* function in $B_1$.
A computation shows that, on the interface $T_1$, we have
\begin{equation}
\label{imp1}
\begin{split}
&G_k(\nabla[P+\psi_k]^{+}, \theta_k \nabla[P+\psi_k]^{-}, \cdot)\\ 
& \quad \leq G_k(\nabla P^{+}, \theta_k \nabla P^{-}, \cdot) 
+ \tcal^{+} (\nabla \psi_k^{+}, \theta_k \nabla \psi_k^{-}; \lambda,\Lambda,\mu)\\
& \qquad \leq G_k(\nabla P^{+}, \theta_k \nabla P^{-}, 0) + \left(|\nabla P^{+}| + |\nabla P^{-}|+1\right)\gamma_k 
+ \tcal^{+} (\nabla \psi_k^{+}, \theta_k \nabla \psi_k^{-}; \lambda,\Lambda,\mu)\\
& \qquad \quad \leq G(\nabla P^{+}, \theta \nabla P^{-}, 0) - |g_k|.
\end{split}
\end{equation}

Let $Q$ be the quadratic polynomial $Q = \frac{\eta}{2(n \Lambda + 4 \mu)}|x- x_0|^2$, and observe that
\[
\begin{cases}
\mcal^{+}(D^2 Q; \frac{\lambda}{n}, \Lambda) \leq \eta & \text{ in } B_1,\\
\tcal^{+}(\nabla Q^{+}, \theta_k \nabla Q^{-}; \lambda, \Lambda, \mu) \leq \eta/2 & \text{ on } T_1.
\end{cases}
\]
Noting that $\|\psi_k\|_{L^{\infty}(B_1)} \to 0$ and $\|v_k - v_{\infty}\|_{L^{\infty}(B_1)} \to 0$ as $k \to \infty$, since $P$ touches $v_{\infty}$ from above in $B_1$, it follows that a vertical translation of $P + \psi_k + Q$ touches $v_{k}$ from above at some $x_k \in B_1$, for sufficiently large $k$.
In fact, $x_k \in T_1$ since otherwise (say) $x_k \in B_1^{+}$ for infinitely many $k$ and hence by the equation for $v_k$
\[
\begin{split}
f^{+}_k(x_k) &
\textstyle
\leq F_k^{+}(D^2 P^{+} + D^2 \psi_k(x_k) + D^2 Q, x_k)\\
& \quad \leq F_k^{+}(D^2 P^{+}, x_k) 
+ \textstyle \mcal^{+}(D^2 \psi_k(x_k); \frac{\lambda}{n}, \Lambda) + \mcal^{+}(D^2 Q; \frac{\lambda}{n}, \Lambda)\\
& \qquad \leq F^{+}(D^2 P^{+}, 0) + \vep_k + \|D^2 P^{+}\| \beta_k(x_k) + \textstyle \mcal^{+}(D^2 \psi_k(x_k); \frac{\lambda}{n}, \Lambda) + \eta\\
& \qquad \quad \leq F^{+}(D^2 P^{+}, 0) - |f_k^{+}(x_k)| +\textstyle \eta,
\end{split}
\]
whence by~\eqref{contra2} we obtain
\[
0 \leq F^{+}(D^2 P^{+}, 0) + \eta \leq - \eta,
\]
a contradiction.

\medskip

Since $x_k \in T_1$ and $\psi_k$ is admissible*, recalling that $v_k$ is a viscosity* solution, by the transmission condition and~\eqref{imp1}, we have that
\[
\begin{split}
g_k(x_k) &
\textstyle
\leq G_k(\nabla [P + \psi_k + Q]^{+}(x_k), \theta_k \nabla [P + \psi_k + Q]^{-}(x_k), x_k)\\
&
\textstyle
\quad \leq G_k(\nabla[P+\psi_k]^{+}(x_k), \theta_k \nabla[P+\psi_k]^{-}(x_k),x_k) + \tcal^{+}(\nabla Q^{+}(x_k), \theta_k \nabla Q^{-}(x_k); \lambda, \Lambda, \mu)\\
& \qquad \leq G(\nabla P^{+}(x_k), \theta \nabla P^{-}(x_k), 0) - |g_k(x_k) | + \eta/2,
\end{split}
\]
whence it follows that
\begin{equation}
\label{lastozko}
0 \leq G(\nabla P^{+}(x_k), \theta \nabla P^{-}(x_k), 0) + \eta/2.
\end{equation}
We wish to point out that this is the only place in the paper where we need the stronger notion of viscosity* solution.

\medskip

Since $x_k \to x_0$ as $k \to \infty$ (recall that $P$ touches $v_{\infty}$ strictly), passing to the limit in~\eqref{lastozko}, by \eqref{contra1} we conclude
\[
0 \leq G(\nabla P^{+}(x_0), \theta \nabla P^{-}(x_0), 0) + \eta/2 \leq - \eta/2,
\]
a contradiction.
Therefore, we must have $G(\nabla P^{+}(x_0), \theta \nabla P^{-}(x_0), 0) \geq 0$ as claimed.
\end{proof}

\begin{rem}
\label{rem:convex}
The key ingredient in the proof above was the use of non-explicit barriers $\psi_k$ defined as solutions to the constant-coefficient, inhomogeneous, and convex transmission problem~\eqref{diri}.
It would be interesting to give conditions under which such barriers become piecewise $C^{1,1}$, i.e., with bounded Hessian up to the interface, and therefore admissible test functions.
Such an improved regularity would allow us to extend Lemma~\ref{lem:approx} to ordinary viscosity solutions.
We point out that the piecewise $C^{1,1}$ regularity of solutions to convex transmission problems is not known even in the homogeneous case.
In the particular case when $G$ is linear and $\theta = 0$, the transmission problem decouples into two boundary-value problems and solutions are piecewise $C^{2,\alpha}$ by the boundary regularity in~\cite{LZ,MS}.
\end{rem}

\medskip

To prove the regularity theorem, we concatenate three simplifying assumptions:

\medskip

\noindent\textbf{First simplification.}
\noindent
Replacing the operators and sources
\[
F^{\pm}(M,x), \qquad G(\xi^{+},\xi^{-},x), \qquad f^{\pm}(x), \qquad g(x)
\]
by the functions
\[
F^{\pm}(M, x) - F^{\pm}(0,x), \qquad G(\xi^{+}, \xi^{-}, x) - G(0,0,0), \qquad f^{\pm}(x) - F^{\pm}(0,x), \qquad g(x) - G(0,0,0),
\]
respectively, we may assume that 
\begin{equation}
\label{simp1}
F^{\pm}(0, x) = 0 \, \text{ in } B_1^{\pm}, \qquad \text{ and } \qquad G(0,0,0) = 0.
\end{equation}

\medskip

\noindent\textbf{Second simplification.}
The modulus $\beta^{\pm}$ is continuous at $0$, hence for each $\delta >0$, there is a small $0 < R_0 \leq 1$ such that
\begin{equation}
\label{mas1}
r^{-1}\|\beta^{\pm}\|_{L^{n}(B_r^{\pm})} \leq \delta \qquad \text{ for all } r \leq R_0.
\end{equation}
Since $u$ solves~\eqref{calarca} in the smaller ball $B_{R_0} \subset B_1$, replacing the data 
\[
u(x), \qquad F^{\pm}(M, x), \qquad G(\xi^{+}, \xi^{-}, x), \qquad f^{\pm}(x), \qquad g(x),
\]
by the rescalings
\[
R_{0}^{-1}u(R_{0} x), 
\qquad 
R_{0} F^{\pm}(R_{0}^{-1} M, R_{0} x),
\qquad
G(\xi^{+}, \xi^{-}, R_{0} x),
\qquad
R_{0} f^{\pm}(R_{0} x),
\qquad g(R_{0} x),
\]
respectively, we may assume that 
\begin{equation}
\label{simp2}
r^{-1}\|\beta^{\pm}\|_{L^{n}(B_r^{\pm})} \leq \delta \quad \text{ for all } r \leq 1.
\end{equation}
Indeed, writing $\widetilde{\beta}^{\pm}$ for the modulus of $R_{0} F^{\pm}(R_{0}^{-1} M, R_{0} x)$, it is easy to see that $\widetilde{\beta}^{\pm}(x) \leq \beta (R_0 x)$ and hence by~\eqref{mas1}
\[
r^{-1}\|\widetilde{\beta}^{\pm}\|_{L^{n}(B_r^{\pm})} 
\leq r^{-1}\|\beta^{\pm}(R_{0} \cdot)\|_{L^{n}(B_r^{\pm})} = (r R_{0})^{-1} \|\beta^{\pm}\|_{L^{n}(B_{r R_{0}}^{\pm})} \leq \delta \quad \text{ for all } r \leq 1.
\]
Also, note that the rescaling preserves the ellipticity constants.

\medskip

\noindent\textbf{Third simplification.}
Further rescaling and normalizing $u$ by
\[
\widetilde{u}(x) := \frac{R^{-1}u(Rx)}{L},
\]
where $0 < R \leq 1$ and
\[
L := \max\left\{
2 \frac{\|u\|_{L^{\infty}(B_{1})}}{R} 
+ 2 \frac{|g(0)|}{\min\{\lambda,1\}}
+ \frac{R^{\alpha}}{\delta}
\left(
[f^{+}]_{C^{\alpha-1}_{[0]}(B_1^{+})} 
+ [f^{-}]_{C^{\alpha-1}_{[0]}(B_1^{-})} 
+ [g]_{C^{\alpha}_{[0]}(T_1)} 
\right), \, 1
\right\},
\]
(here recall definitions~\eqref{point:morrey} and~\eqref{point:holder} for the punctual norms)
the original problem~\eqref{calarca} restricted to the smaller ball $B_{R} \subset B_{1}$ can be written equivalently as
\[
\begin{cases}
\widetilde{F}^{\pm}(D^2 \widetilde{u}, x) := 
\frac{R}{L}
F^{\pm}\left(
\frac{L}{R}D^2 \widetilde{u}, R x\right) = \frac{R}{L} f^{\pm}\left(R x\right) =: \widetilde{f}^{\pm}(x)
& \text{ in } B^{\pm}_1,\\
\widetilde{G}(\nabla \widetilde{u}^{+}, \theta \nabla \widetilde{u}^{-}, x) := 
\frac{1}{L}
G\left( L \nabla \widetilde{u}^{+}, L \theta \nabla \widetilde{u}^{-}, R x\right)
= \frac{1}{L} g\left(R x\right) =: \widetilde{g}(x)
& \text{ on } T_1.
\end{cases}
\]
The new operators $\{\widetilde{F}^{\pm}, \widetilde{G}\}$ have ellipticity constants $\lambda$, $\Lambda$, and $\mu$.
By definition, we have that
\begin{equation}
\label{simp3}
\|\widetilde{u}\|_{L^{\infty}(B_1)} \leq \frac{1}{2} \qquad \text{ and } \qquad |\widetilde{g}(0)| \leq \frac{\min\{1, \lambda\}}{2}.
\end{equation}
Since $\widetilde{\beta}^{\pm}(x) := \beta_{\widetilde{F}^{\pm}}(x) \leq \beta^{\pm}(R x)$, arguing as above, it is easy to see that~\eqref{simp2} still holds for $\widetilde{\beta}^{\pm}$.
Moreover, since $\widetilde{\gamma}(x) := \gamma_{\widetilde{G}}(x) \leq \gamma(R x)$, 
choosing $0 <R \leq 1$ small such that 
\[
R \leq \left(\delta [\gamma]_{C^{\alpha}_{[0]}(T_1)}^{-1}\right)^{\frac{1}{\alpha}},
\]
we obtain
\begin{equation}
\label{simp4}
\|\widetilde{\gamma}\|_{L^{\infty}(T_r)} 
\leq \|\gamma\|_{L^{\infty}(T_{r R})}
\leq [\gamma]_{C^{\alpha}_{[0](T_1)}} R^{\alpha} r^{\alpha}
\leq \delta r^{\alpha}
\qquad \text{ for all } r \leq 1.
\end{equation}
Finally, we note that the forcing terms are small with
\begin{equation}
\label{simp5}
\|\widetilde{f}^{\pm}\|_{L^n(B_r^{\pm})}
\leq \frac{1}{L} \|f^{\pm}\|_{L^n(B^{\pm}_{r R})}
\leq \frac{[f^{\pm}]_{C^{\alpha-1}_{[0]}(B_1^{\pm})}R^{\alpha}}{L} 
r^{\alpha}
\leq \delta r^{\alpha} \qquad \text{ for all } r \leq 1,
\end{equation}
and
\begin{equation}
\label{simp6}
\|\widetilde{g} - \widetilde{g}(0)\|_{L^{\infty}(T_r)} = \frac{1}{L} \|g - g(0)\|_{L^{\infty}(T_{rR})}
\leq  \frac{[g]_{C^{\alpha}_{[0]}(T_1)}R^{\alpha}}{L} r^{\alpha}
\leq \delta r^{\alpha} \qquad \text{ for all } r \leq 1.
\end{equation}

\medskip

Thanks to the simplifying assumptions~\eqref{simp1}, \eqref{simp2}, \eqref{simp3}, \eqref{simp4}, \eqref{simp5}, and \eqref{simp6} above, 
it suffices to prove the following pointwise estimate:

\begin{thm}
Let $\{F^{\pm}, G\}$ be operators with ellipticity constants $0 < \lambda \leq \Lambda$ and $\mu \geq 0$,
with H\"{o}lder continuous moduli $\beta^{\pm}$ and $\gamma$ in $B_1^{\pm} \cup T_1$ and $T_1$, respectively.
Let $f^{\pm}$ and $g$ be H\"{o}lder continuous functions in $B_1^{\pm} \cup T_1$ and $T_1$, respectively.
Let $\theta \in [0,1]$ be a constant.
Let $u \in C(B_1)$ be a viscosity solution to 
\[
\begin{cases}
F^{\pm}(D^2 u, x) = f^{\pm}(x) & \text{ in } B_{1}^{\pm},\\
G(\nabla u^{+}, \theta \nabla u^{-}, x) = g(x) &\text{ on } T_{1}.
\end{cases}
\]
Assume that 
\begin{equation}
\label{ue1}
F^{\pm}(0, x) = 0 \, \text{ in } B_1^{\pm}, \qquad G(0,0, 0) = 0, \qquad \|u\|_{L^{\infty}(B_1)} \leq \frac{1}{2}, \qquad |g(0)| \leq \frac{\min\{1,\lambda\}}{2}.
\end{equation}

For each $0 < \alpha < \overline{\alpha}$, there exists $\delta > 0$ depending only on $n$, $\lambda$, $\Lambda$, $\mu$, $\alpha$, and $\overline{\alpha}$ such that if
\begin{equation}
\label{ue2}
\|\beta^{\pm}\|_{L^n(B_r^{\pm})} \leq \delta r,
\quad 
\|\gamma\|_{L^{\infty}(T_r)} \leq \delta r^{\alpha},
\quad
\|f^{\pm}\|_{L^n(B_r^{\pm})} \leq \delta r^{\alpha},
\quad
\|g- g(0)\|_{L^{\infty}(T_r)} \leq \delta r^{\alpha}
\end{equation}
for all $r \leq 1$,
then $u$ is piecewise $C^{1,\alpha}$ at the origin, i.e., there is a piecewise linear function
\[
\ell(x) = a^{+} \cdot x \, \chara{\{x_n \geq 0\}} + a^{-} \cdot x \, \chara{\{x_n < 0\}} + u(0),
\]
with $a^{\pm} \in \R^n$, such that
\[
|a^{\pm}| \leq C,
\]
\[
G(a^{+}, \theta a^{-}, 0) = g(0),
\]
\[
\|u - \ell\|_{L^{\infty}(B_r)} \leq 
Cr^{1 +\alpha}
\quad \text{ for all } r \leq 1,
\]
for some constant $C > 0$ depending only on $n$, $\lambda$, $\Lambda$, $\mu$, $\alpha$, and $\overline{\alpha}$.
\end{thm}
\begin{proof}
First, we recall the piecewise $C^{1,\overline{\alpha}}$ estimate~\eqref{fund2} stated in the previous section, as well as the following variant of the H\"{o}lder estimate (Theorem~\ref{thm:holder}):
\begin{itemize}
\item 
If $v \in S^{\star}(\lambda, \Lambda, \mu, \theta; f^{+}, f^{-}, g)$ in $B_{1}$, then $v$ is $C^{\mao}$ in $B_1$, with an estimate
\begin{equation}
\label{fund1}
[v]_{C^{\mao}(\overline{B}_{3/4})} \leq C_{\star}\left( 
\|v\|_{L^{\infty}(B_{1})} + \|f^{+}\|_{L^{n}(B_{1}^{+})} + \|f^{-}\|_{L^{n}(B_{1}^{-})} + \|g\|_{L^{\infty}(T_{1})}
\right),
\end{equation}
where $0 < \mao < 1$ and $C_{\star} > 0$ are universal constants.
\end{itemize}

\medskip

Our efforts are devoted to constructing a sequence of piecewise linear functions
\begin{equation}
\label{g0}
\ell_k(x) = a_k^{+} \cdot x \, \chara{\{x_n \geq 0\}} + a_k^{-} \cdot x \, \chara{\{x_n < 0\}} + c_k,
\end{equation}
with $a_k^{\pm} \in \R^n$ and $c_k \in \R$ such that, for all $k \geq 0$
\begin{equation}
\label{g1}
G(a_k^{+}, \theta a_k^{-}, 0) = g(0),
\end{equation}
\begin{equation}
\label{g2}
\|u - \ell_k\|_{L^{\infty}(B_{\rho^{k}})} \leq \rho^{k(1+\alpha)},
\end{equation}
\begin{equation}
\label{g3}
|a_k^{\pm} - a_{k-1}^{\pm}| \leq C_0 \rho^{(k-1)\alpha}, \qquad |c_k - c_{k-1}| \leq C_0 \rho^{(k-1)(1+\alpha)},
\end{equation}
\begin{equation}
\label{g4}
|(u - \ell_k)(\rho^{k} x) - (u - \ell_k)(\rho^{k} y)| \leq 2 C_{\star} \rho^{k (1+ \alpha)} |x-y|^{\mao}
\quad \text{ for all } 
x \text{ and } y \text{ in } \overline{B}_{3/4}.
\end{equation}
where $0 < \rho < 1/3$ is a constant depending only on $n$, $\lambda$, $\Lambda$, $\mu$, $\alpha$, and $\overline{\alpha}$.
We can then take $\ell = \lim_{k} \ell_k$ to be the piecewise linear function from the statement.

\medskip

Let $\ell_0 = \ell_{-1} = t (x_n)_{+}$, where $t \in \R$ is the unique constant such that
\begin{equation}
\label{deo1}
G(t e_n, 0,0) = g(0)
\end{equation}
and whose existence is guaranteed by ellipticity of $G$.

\medskip

We claim that~\eqref{g0}, \eqref{g1}, \eqref{g2}, \eqref{g3}, and~\eqref{g4} hold for $k = 0$.
Indeed, first note that~\eqref{g1} is equivalent to~\eqref{deo1}.
By~\eqref{ue1}, we have that
\begin{equation}
\label{t:bound}
|t| \leq \frac{1}{\lambda} |g(0)| \leq \frac{1}{2},
\end{equation}
whence it follows that
\[
\|u - t x_n^{+}\|_{L^{\infty}(B_1)} \leq \|u\|_{L^{\infty}(B_1)} + |t| \leq 1
\]
and hence~\eqref{g2}.
Since~\eqref{g3} is obvious, it remains to prove~\eqref{g4}.
For this, noting that $u \in S(\lambda, \Lambda, \mu, \theta; f^{+}, f^{-}, g - G(0,0,\cdot))$ in $B_1$ 
(see Proposition~2.16 in~\cite{E}),
by the $C^{\mao}$ estimate~\eqref{fund1} combined with~\eqref{ue1} and~\eqref{ue2}, we deduce
\[
\begin{split}
[u]_{C^{\mao}(\overline{B}_{3/4})}
&\leq C_{\star} \left( 
\|u\|_{L^{\infty}(B_1)} 
+ \|f^{+}\|_{L^n(B_1^{+})}
+ \|f^{-}\|_{L^n(B_1^{-})}
+ \|g - G(0,0,\cdot)\|_{L^{\infty}(T_1)}
\right)
\\
& \quad \leq C_{\star} \left( 
\frac{1}{2}
+ 2 \delta 
+ |g(0)| 
+ \|g - g(0)\|_{L^{\infty}(T_1)}
+ \|G(0,0,\cdot)\|_{L^{\infty}(T_1)}
\right)\\
& \qquad \leq C_{\star} \left( 
1 + 3 \delta + \|\gamma\|_{L^{\infty}(T_1)}\right) 
\leq C^{\star} \left( 1+4\delta\right) \leq 2 C^{\star}
\end{split}
\]
by taking $\delta \leq 1/4$, whence~\eqref{g4} follows.

\medskip

Assume by induction that we have constructed $\ell_k$ given by~\eqref{g0} and satisfying~\eqref{g0}, \eqref{g1}, \eqref{g2}, \eqref{g3}, and~\eqref{g4} for some $k \geq 0$.
We must build $\ell_{k+1}$ and check that the above properties hold for $k+1$.

\medskip

First, we note a useful coefficient bound.
Since $\rho \leq 1/3$, from~\eqref{g3} and~\eqref{t:bound} we deduce
\begin{equation}
\label{a:bound}
|a_k^{\pm}| \leq |a^{\pm}_0| + \sum_{j = 1}^{k} |a_j^{\pm} - a_{j-1}^{\pm}| \leq 
\frac{1}{2} + C_0 \sum_{j = 1}^{\infty}\rho^{(j-1) \alpha} 
\leq \frac{1}{2} + \frac{C_0}{1 - 3^{- \alpha}} =: C_1.
\end{equation}

Let $v$ be the function
\[
v(x) := \frac{(u-\ell_k)(\rho^{k}x)}{\rho^{k(1+\alpha)}}
\]
and note that it is a viscosity solution to the transmission problem
\begin{equation}
\label{k0}
\begin{cases}
F^{\pm}_k(D^2 v, x) = f_k^{\pm}(x) & \text{ in } B_1^{\pm},\\
G_k(\nabla v^{+}, \theta \nabla v^{-}, x) = g_k(x) & \text{ on } T_1,
\end{cases}
\end{equation}
where
\[
F_k^{\pm}(M, x) := \rho^{k(1-\alpha)} F^{\pm}(\rho^{- k(1-\alpha)} M, \rho^k x),
\qquad f_k^{\pm}(x) := \rho^{k(1-\alpha)} f^{\pm}(\rho^{k} x),
\]
\[
G_k (\xi^{+}, \xi^{-}, x) := \frac{1}{\rho^{k \alpha}} \left[
G(a_k^{+} + \rho^{k\alpha} \xi^{+}, \theta a_k^{-} + \rho^{k\alpha} \xi^{-}, \rho^{k} x)
- G(a_k^{+}, \theta a_k^{-}, \rho^{k} x)
\right],
\]
and
\[
g_k(x) := \frac{1}{\rho^{k \alpha}} \left[g(\rho^{k} x)
- G(a_k^{+}, \theta a_k^{-}, \rho^{k} x)
\right].
\]

\medskip

The operators $\{F_{k}^{\pm}, G_k\}$ have ellipticity constants $\lambda$, $\Lambda$, and $\mu$, and by definition satisfy
\begin{equation}
\label{k1}
F_k^{\pm}(0, x) = 0 \, \text{ in } B_1^{\pm}, \qquad G_k(0,0,x) = 0 \, \text{ on } T_1.
\end{equation}
The sources $f^{\pm}_{k}$ and $g_k$ are small, namely, by the induction hypothesis~\eqref{g1} we have that
\begin{equation}
\label{k2}
g_k(0) =
\rho^{-k \alpha}
\left[g(0)
- G(a_k^{+}, \theta a_k^{-}, 0)
\right]
= 0
\end{equation}
and hence by~\eqref{ue2}
\begin{equation}
\label{k3}
\|f_k^{\pm}\|_{L^n(B_1^{\pm})} = \rho^{-k\alpha} \|f^{\pm}\|_{L^{n}(B^{\pm}_{\rho^{k}})} \leq \delta,
\end{equation}
and further by the coefficient bounds~\eqref{a:bound}
\begin{equation}
\label{k4}
\begin{split}
 \|g_k\|_{L^{\infty}(T_1)} &\leq \rho^{-k\alpha} \|g - g(0)\|_{L^{\infty}(T_{\rho^{k}})} 
 + \rho^{-k \alpha} \| G(a_k^{+}, \theta a_k^{-}, \rho^{k} \cdot) - G(a_k^{+}, \theta a_k^{-}, 0)\|_{L^{\infty}(T_{\rho^{k}})}\\
& \qquad \leq \delta + (|a_k^{+}| + |a_k^{-}| + 1) \rho^{-k \alpha} \|\gamma\|_{L^{\infty}(T_{\rho^{k}})}\\
&  \qquad \qquad \leq 2 (C_1 +1 ) \delta.
 \end{split}
\end{equation}

\medskip

Moreover, the operators are close to constant-coefficient ones.
For $x \in B_1^{\pm}$ we have that
\[
\begin{split}
&\big|F_k^{\pm}(M, x) - F_k^{\pm}(M, 0)\big|\\
& \qquad \leq \rho^{k(1-\alpha)} \big|F^{\pm}(\rho^{-k(1-\alpha)}M, \rho^{k} x) - F^{\pm}(\rho^{-k(1-\alpha)}M, 0)\big|\\
& \qquad \qquad \leq \beta(\rho^k x) \left( \|M\| + 1\right)
\end{split}
\]
and hence by~\eqref{ue2} we obtain
\begin{equation}
\label{k5}
\|\beta_{F_k^{\pm}}\|_{L^n(B_1^{\pm})} \leq \rho^{-k}\|\beta^{\pm}\|_{L^n(B^{\pm}_{\rho^{k}})} \leq \delta.
\end{equation}
Similarly, for $x \in T_1$ we see that
\[\begin{split}
&\big|G_k(\xi^{+}, \xi^{-}, x) - G_k(\xi^{+}, \xi^{-}, 0) \big| \\
& \qquad \leq 
\rho^{-k \alpha}\big|G(a_k^{+} + \rho^{k \alpha} \xi^{+},\theta a_k^{-} + \rho^{k \alpha} \xi^{-}, \rho^{k} x) -G(a_k^{+} + \rho^{k\alpha} \xi^{+}, \theta a_k^{-} + \rho^{k\alpha} \xi^{-}, 0) \big|\\ 
& \qquad \qquad +\rho^{-k \alpha} \big|G(a_{k}^{+}, \theta a_{k}^{-}, \rho^{k} x) - G(a_{k}^{+}, \theta a_{k}^{-},0)\big|\\
& \qquad \qquad \qquad \leq 
\left( \big| a_{k}^{+} + \rho^{k\alpha} \xi^{+}\big| + \big| \theta a_{k}^{-} + \rho^{k\alpha} \xi^{-}\big| + |a_{k}^{+}| + \theta |a_{k}^{-}| + 2 \right) 
\rho^{-k \alpha} \gamma(\rho^{k} x)\\
& \qquad \qquad \qquad \qquad  \leq \left( |\xi^{+}| + |\xi^{-}| + 2 |a_k^{+}| + 2 |a_k^{-}| + 2\right) 
\rho^{-k \alpha} \|\gamma\|_{L^{\infty}(T_{\rho^{k}})}\\
& \qquad \qquad \qquad \qquad \qquad \leq 2(2 C_1 + 1) \left( |\xi^{+}| + |\xi^{-}| + 1\right) 
\delta,
\end{split}\]
where in the last line we have used~\eqref{ue2} and~\eqref{a:bound}, and hence
\begin{equation}
\label{k6}
\|\gamma_{G_k}\|_{L^{\infty}(T_1)} \leq 2(2 C_1 + 1) \delta.
\end{equation}

\medskip

Next, we approximate $v$ by a constant-coefficient replacement.
Let $w$ be the unique viscosity solution to 
\begin{equation}
\label{rep1}
\begin{cases}
F_k^{\pm}(D^2 w, 0) = 0 & \text{ in } B_{3/4},\\
G_k(\nabla w^{+}, \theta \nabla w^{-}, 0) = 0 & \text{ in } T_{3/4},\\
w = v & \text{ on } \partial B_{3/4},
\end{cases}
\end{equation}
and fix $\vep > 0$.
By~\eqref{g4}, we have that
\[
[v]_{C^{\mao}(\overline{B}_{3/4})} \leq 2C_{\star},
\]
hence the boundary datum $\phi = v|_{\partial B_{3/4}}$ has a H\"{o}lder modulus of continuity 
$\omega(r) = 2 C_{\star} r^{\mao}$.
Since $v$ solves~\eqref{k0}, 
with~\eqref{k1}, \eqref{k2}, and small data satisfying~\eqref{k3}, \eqref{k4}, \eqref{k5}, and~\eqref{k6},
the approximation lemma (Lemma~\ref{lem:approx}) yields the estimate
\begin{equation}
\label{approx}
\|v-w\|_{L^{\infty}(B_{3/4})} \leq \vep,
\end{equation}
by choosing $\delta > 0$ sufficiently small depending only on 
$n$, $\lambda$, $\Lambda$, $\mu$, and $\vep$.

\medskip

We claim that~\eqref{rep1} has interior piecewise $C^{1,\overline{\alpha}}$ estimates of the form~\eqref{fund2}.
Indeed, the equation can be written equivalently as
\[
\begin{cases}
\widetilde{F}^{\pm}(D^2 [\rho^{k\alpha} w]) := 
\rho^{k(2-\alpha)}F^{\pm}(\rho^{-k(2-\alpha)} D^2 [\rho^{k\alpha} w], 0) = 0 & \text{ in }B_{3/4}^{\pm},\\
\widetilde{G}(\nabla[\rho^{k\alpha} w^{+}], \theta \nabla[\rho^{k\alpha} w^{-}])
:= G(a_k^{+} + \nabla[\rho^{k\alpha} w^{+}], \theta a_k^{-} + \theta \nabla[\rho^{k \alpha} w^{-}], 0) - g(0) = 0 & \text{ on } T_{3/4},
\end{cases}
\]
where the new operators $\{\widetilde{F}^{\pm}, \widetilde{G}\}$ have ellipticity constants $\lambda$, $\Lambda$, and $\mu$, and satisfy $\widetilde{F}^{\pm}(0) = \widetilde{G}(0, 0) = 0$ by~\eqref{ue1} and the induction hypothesis~\eqref{g1}.
Hence, by interior regularity for such problems, we have the estimate
\begin{equation}
\label{rep2}
\|w^{\pm}\|_{C^{1,\overline{\alpha}}(\overline{B}^{\pm}_{1/2})} \leq C_0 \|w\|_{L^{\infty}(B_{3/4})}
\leq C_0 \|v\|_{L^{\infty}(\partial B_{3/4})} \leq C_0,
\end{equation}
where in the last inequalities we have used the ABP and that $\|v\|_{L^{\infty}(\partial B_{3/4})} \leq 1$ by~\eqref{g2}.

\medskip

Let $\overline{\ell}$ be the piecewise linear function given on each side by
\[
\overline{\ell}^{\pm}(x) := \nabla w^{\pm}(0) \cdot x + w(0).
\]
From~\eqref{rep2}, we deduce 
\begin{equation}
\label{rep3}
\|w - \overline{\ell}\|_{L^{\infty}(B_{3\rho/2})} 
\leq C_0 \left(\frac{3}{2} \rho\right)^{1+\overline{\alpha}}
\leq \frac{1}{2} \rho^{1+\alpha}
\end{equation}
by choosing $0 < \rho< 1/3$  sufficiently small depending on $n$, $\lambda$, $\Lambda$, $\mu$, and $\alpha < \overline{\alpha}$.
Letting $\vep = \frac{1}{2} \rho^{\alpha}$, combining~\eqref{approx} and~\eqref{rep3} yields
\begin{equation}
\label{rep4}
\begin{split}
\|v - \overline{\ell}\|_{L^{\infty}(B_{\rho})} 
\leq \|v - \overline{\ell}\|_{L^{\infty}(B_{3\rho/2})} 
\leq \|v-w\|_{L^{\infty}(B_{3/4})} + \|w - \overline{\ell}\|_{L^{\infty}(B_{3\rho/2})} 
\leq \rho^{1+\alpha}.
\end{split}
\end{equation}

\medskip

Define the $(k+1)$-th piecewise linear function by 
\[
\ell_{k+1}(x) := \ell_k(x) + \rho^{k (1+ \alpha)} \textstyle\overline{\ell}\left(\rho^{-k} x \right),
\]
or equivalently, with coefficients
\[
a^{\pm}_{k+1} := a^{\pm}_k + \rho^{k\alpha} \nabla w^{\pm}(0), \qquad c_{k+1} := c_k + \rho^{k(1+\alpha)} w(0).
\]
We check that $\ell_{k+1}$ satisfies~\eqref{g1}, \eqref{g2}, \eqref{g3}, and \eqref{g4} with $k+1$.
Since $w$ satisfies the transmission condition from~\eqref{rep1} in the classical sense, evaluating it at $0$ we see that
\[
G_k(\nabla w^{+}(0), \theta \nabla w^{-}(0), 0) = 0,
\]
which is equivalent to~\eqref{g1} by definition of $G_k$.
Note that~\eqref{rep4} can be written as 
\[
\|u - \ell_{k+1}\|_{L^{\infty}(B_{\rho^{k+1}})} 
\leq \|u - \ell_{k+1}\|_{L^{\infty}(B_{\frac{3}{2}\rho^{k+1}})} 
\leq \rho^{(k+1)(1+\alpha)},
\]
whence~\eqref{g2} follows.
By~\eqref{rep2}, we also have the estimates
\[
|a_{k+1}^{\pm} - a_k^{\pm}| = \rho^{k \alpha} |\nabla w^{\pm}(0)| 
\leq C_0 \rho^{k\alpha}
\qquad \text{ and } \qquad
|c_{k+1} - c_k|
= \rho^{k (1+\alpha)} |w(0)|
\leq C_0 \rho^{k(1+\alpha)},
\]
giving~\eqref{g3}.
Finally, we observe that
\[
\frac{(u - \ell_{k+1})(\rho^{k+1}x)}{\rho^{(k+1)(\alpha+1)}} = \frac{(v - \overline{\ell})(\rho x)}{\rho^{1+\alpha}} \quad \text{ for } x \in \overline{B}_{3/2},
\]
and hence~\eqref{g4} is equivalent to
\begin{equation}
\label{g5}
[v - \overline{\ell}]_{C^{\mao}(\overline{B}_{3\rho/4})} \leq 2 C_{\star} \rho^{1+\alpha-\mao}.
\end{equation}
Here, note that $v - \overline{\ell}$ is a viscosity solution to the transmission problem
\[
\begin{cases}
F^{\pm}_k(D^2 (v - \overline{\ell}), x) = f^{\pm}_k(x) & \text{ in } B_{3\rho/2}^{\pm},\\
G_k(\nabla w^{+}(0) + \nabla (v - \overline{\ell})^{+}, \theta \nabla w^{-}(0) + \theta \nabla (v - \overline{\ell})^{-}, x) = g_k(x) & \text{ on } T_{3\rho/2},
\end{cases}
\]
whence
\[
v - \overline{\ell} \in S(\lambda, \Lambda, \mu, \theta; f_k^{+}, f_k^{-}, g_k- G_k(\nabla w^{+}(0), \theta \nabla w^{-}(0), \cdot))
\qquad \text{ in } B_{3\rho/2}
\]
and invoking the $C^{\mao}$ estimate~\eqref{fund1}, we obtain
\begin{equation}
\label{g6}
\begin{split}
&\rho^{\mao} [v - \overline{\ell}]_{C^{\mao}(\overline{B}_{\rho})}\\
& \quad \leq 
C_{\star}\|v- \overline{\ell}\|_{L^{\infty}(B_{3\rho/2})} 
+ C_{\star}\rho \Big( \|f_k^{+}\|_{L^{n}(B_{1/2}^{+})} + \|f_k^{-}\|_{L^{n}(B_{1/2}^{-})}+ \|g_k\|_{L^{\infty}(T_{1/2})} \Big)\\
& \qquad \quad + C_{\star} \rho \big\|G_k(\nabla w^{+}(0), \theta \nabla w^{-}(0), \cdot) - G_k(\nabla w^{+}(0), \theta \nabla w^{-}(0), 0)\big\|_{L^{\infty}(T_{1/2})}\\
& \qquad \qquad \leq C_{\star} \left( \rho^{1+\alpha} + \rho \left[ 
2 \delta + 2 (C_1 +1) \delta 
+ (|\nabla w^{+}(0)| + |\nabla w^{-}(0)| + 1)\|\gamma_{G_k}\|_{L^{\infty}(T_{1/2})}
\right]\right)\\
& \qquad \qquad \quad \leq C_{\star} \left( \rho^{1+\alpha}
+ 2 \rho \delta \left[2 + C_1 + (2 C_0 + 1) (2 C_1 + 1) 
\right]\right),
\end{split}
\end{equation}
where we have used the bounds~\eqref{rep4}, \eqref{ue2}, and~\eqref{rep2}.
Choosing $\delta > 0$ smaller such that
\[
\delta \leq \frac{\rho^{\alpha}}{2 \left[2 + C_1 + (2 C_0 + 1)(2 C_1 + 1) 
\right]},
\]
by~\eqref{g6}, we obtain~\eqref{g5}.
\end{proof}

\medskip

\appendix

\section{Gluing moduli of continuity up to the boundary}
\label{app:boundary}

Here, we give the technical details to complete the proof of Theorem~\ref{thm:barriers}.

\medskip

Let $u$ be a function in $C(\overline{B}_{1})$ with the following three types of moduli of continuity:
\begin{itemize}
\item A modulus $\omega$ such that 
\[
|u(x_0) - u(y_0)| \leq \omega(|x_0 - y_0|) \qquad \text{ for all } x_0 \text{ and } y_0 \text{ on }  \partial B_1.
\]
\item For each $\kappa > 0$, a modulus $\omega_{\kappa}$ such that
\[
|u(x) - u(x_0)| \leq \omega_{\kappa}(|x-x_0|) \qquad \text{ for all } x \in B_1 \text{ and } x_0 \in \{|x_n| \geq \kappa\} \cap \partial B_1;
\]
\item A modulus $\omega_0$ such that
\[
|u(x) - u(x_0)| \leq \omega_{0}(|x-x_0|) \qquad \text{ for all } x \in B_1 \text{ and } x_0 \in \{x_n = 0 \} \cap \partial B_1;
\]
\end{itemize}

\medskip

Note here that $\omega_{\kappa}$ could degenerate as $\kappa \downarrow 0$.

\medskip

In the proofs below, we denote the projection of $x \neq 0$ onto the sphere $\partial B_1$ by $\pi[x] = \frac{x}{|x|}$.
We also write $d_x = \textrm{dist}(x, \partial B_1) = |x - \pi[x]| = 1 -|x|$ for the distance.

\medskip

\begin{lem}
There is a modulus of continuity $\widetilde{\omega}$ depending only on $\omega$, $\{\omega_{\kappa}\}_{\kappa > 0}$, and $\omega_0$ such that
\[
|u(x) - u(x_0)| \leq \widetilde{\omega}(|x-x_0|) \quad \text{ for all } x \in \overline{B}_{1} \text{ and } x_0 \in \partial B_1.
\]
\end{lem}

\begin{proof}
Let $\vep > 0$.
We must find a $\delta > 0$ depending only on $\vep$ and the given moduli, such that
\begin{equation}
\label{tri0}
|u(x) - u(x_0)| \leq \vep \quad \text{ whenever } |x - x_0| \leq \delta, \text{ with } x \in B_1 \text{ and } x_0 \in \partial B_1.
\end{equation}

\medskip

Assume that $|x - x_0| \leq \delta$, with $0 < \delta \leq 1/2$ small to be determined below.
By triangle inequality
\begin{equation}
\label{tri1}
|u(x) - u(x_0)| \leq |u(x) - u(\pi[x])| + |u(\pi[x]) - u(x_0)|.
\end{equation}

Since $|x| \geq |x_0| - |x- x_0| \geq 1/2$ and $d_x \leq |x- x_0|$, we have  that
\[
|\pi[x] - x_0 | = 
\left|\frac{x}{|x|} - x_0\right| 
= \frac{1}{|x|}\big|x - x_0 + (1-|x|) x_0\big|
\leq \frac{|x- x_0|}{|x|} +\frac{d_x}{|x|}
\leq 4 \delta
\]
and the second term on the right-hand side of~\eqref{tri1} can be controlled by the modulus on $\partial B_1$ as
\begin{equation}
\label{tri2}
|u(\pi[x]) - u(x_0)| \leq \omega(4 \delta).
\end{equation}

\medskip

To bound the first term, fix $0 < \kappa_0 \leq 1/4$ small to be chosen below. 
We distinguish two cases:
\begin{itemize}
\item If $|x_n| \geq \kappa_0$, then by the modulus $\omega_{\kappa_0}$, we have
\begin{equation}
\label{tri3}
|u(x) - u(\pi[x])| \leq \omega_{\kappa_0}(d_x) \leq \omega_{\kappa_0}(\delta).
\end{equation}
\item If $|x_n| \leq \kappa_0$, then we have the additional bounds
\begin{equation}
\label{eq:bod1}
|x'| \geq |x| - |x_n| \geq 1/2 - \kappa_0 \geq 1/4,
\end{equation}
and
\begin{equation}
\label{eq:bod2}
1 - |x'| \leq 1 - |x| + |x_n| < \delta + \kappa_0 \leq 2 \kappa_0.
\end{equation}
By triangle inequality, using the modulus $\omega_0$
\begin{equation}
\label{bodo}
\begin{split}
|u(x) - u(\pi[x])| &\leq \textstyle \left|u(x) - u(\frac{x'}{|x'|},0)\right| + \left| u(\pi[x]) - u(\frac{x'}{|x'|},0)\right|\\
& \quad \leq \textstyle 
\omega_0\left(\left|x -(\frac{x'}{|x'|},0) \right|\right) +  \omega_0\left(\left|\frac{x}{|x|} - (\frac{x'}{|x'|},0) \right| \right).
\end{split}
\end{equation}
Note that by~\eqref{eq:bod2}, we have
\begin{equation}
\label{cris1}
\textstyle
\left|x -(\frac{x'}{|x'|},0) \right| = \sqrt{(1- |x'|)^2 +x_n^2}
\leq \sqrt{5}\kappa_0 < 2 \sqrt{2} \kappa_0.
\end{equation}
On the other hand, by concavity of the square root
\[
|x| = \sqrt{|x'|^2+x_n^2} \leq |x'| + \frac{1}{2 |x'|}x_n^2,
\]
whence, by~\eqref{eq:bod1}, we deduce
\begin{equation}
\label{cris2}
\textstyle
\left|\frac{x}{|x|} - (\frac{x'}{|x'|},0) \right| = \sqrt{2}\sqrt{\frac{|x|-|x'|}{|x|}}
\leq \frac{|x_n|}{\sqrt{|x||x'|}}
\leq 2 \sqrt{2} \kappa_0.
\end{equation}
Applying~\eqref{cris1} and~\eqref{cris2} in~\eqref{bodo}, we conclude that
\begin{equation}
\label{tri4}
\begin{split}
|u(x) - u(\pi[x])| 
\leq 2 \omega_0(2 \sqrt{2} \kappa_0).
\end{split}
\end{equation}
\end{itemize}

\medskip

To finish the argument,
first choose $\kappa_0 > 0$ such that 
\[
2 \omega_0(2 \sqrt{2}\kappa_0) \leq \vep/2
\]
and then choose $\delta > 0$ such that
\[
\max\{ \omega(4\delta), \omega_{\kappa_0}(\delta) \} \leq \vep/2.
\]
Combining~\eqref{tri2},~\eqref{tri3},~and~\eqref{tri4} to control~\eqref{tri1} now yields the claim~\eqref{tri0}.
\end{proof}

\medskip

Assume additionally that $u$ satisfies the following interior H\"{o}lder estimate in $B_1$:
\[
R^{\beta} [u - c]_{C^{\beta}(\overline{B}_{R/2}(x))} \leq 
C \left( 
\|u - c\|_{L^{\infty}(B_R(x))} + R K
\right) 
\quad
\text{ for all } c \in \R \text{ and } B_R(x) \subset B_1,
\]
for some constants $0 < \beta < 1$, $C > 0$, and $K \geq 0$.
Solutions to transmission problems satisfy such an estimate by Corollary~\ref{cor:trueint}, with $\beta$ and $C$ universal and $K = \|f^{+}\|_{L^n(B_1^{+})} + \|f^{-}\|_{L^n(B_1^{-})} + \|g\|_{L^{\infty}(T_1)}$.

\medskip

\begin{lem}
There is a modulus of continuity $\omega^{\star}$ depending only on
$\widetilde{\omega}$, $\beta$, $C$, and $K$ such that
\[
|u(x) - u(y)| \leq \omega^{\star}(|x-y|) \qquad \text{ for all } x \text{ and } y \text{ in } \overline{B}_{1}.
\]
\end{lem}
\begin{proof}
Let $\vep > 0$.
We will find a $\delta > 0$ depending only on $\vep$ and the above data such that
\[
|u(x) - u(y)| \leq \vep \quad \text{ whenever } |x- y| \leq \delta.
\]
Let $|x- y| \leq \delta$, with $0 < \delta \leq 1/4$ small to be determined below.
There are two cases:

\medskip

\noindent
\textbf{Case~1.} $\max\{d_x, d_y\} \leq 2 |x-y|$.

\medskip

Since $|x| = 1- d_x \geq 1 - 2 \delta \geq 1/2$, we have that
\[
\big|\pi[x] - \pi[y]\big| = 
\left|\frac{x}{|x|} - \frac{y}{|y|}\right| 
\leq \frac{2}{|x|} |x-y|
\leq 4 |x-y| \leq 4 \delta,
\]
and by the boundary modulus
\[
\begin{split}
|u(x) - u(y)| &\leq 
|u(x) - u(\pi[x])| + |u(\pi[x]) - u(\pi[y])| +|u(y) - u(\pi[y])| \\
& \quad \leq 
\widetilde{\omega}(d_x) + \widetilde{\omega}(|\pi[x]-\pi[y]|) + \widetilde{\omega}(d_y)\\
& \qquad \leq 
3 \widetilde{\omega}(4 \delta)\\
& \qquad \quad \leq \vep
\end{split}
\]
for $\delta > 0$ sufficiently small.

\medskip

\noindent
\textbf{Case~2.} $d_x \geq 2 |x-y|$.

\medskip

Since $B_{2|x-y|}(x) \subset B_{d_x}(x) \subset B_1$, by the interior estimate, we have that
\[
d_x^{\beta} [u- u(\pi[x])]_{C^{\beta}(\overline{B}_{d_x/2}(x))}
\leq C \left( 
\|u - u(\pi[x])\|_{L^{\infty}(B_{d_x}(x))}
+ d_x K
\right)
\]
and applying the boundary modulus
\[
\|u - u(\pi[x])\|_{L^{\infty}(B_{d_x}(x))} \leq \widetilde{\omega}(d_x),
\]
we conclude
\[
|u(x)-u(y)| \leq
C \left( \widetilde{\omega}(d_x) d_x^{-\beta} + d_x^{1-\beta} K \right) |x-y|^{\beta}.
\]
Let $0 < \delta_0 < 1$ small to be chosen next.
There are two subcases:
\begin{itemize}
\item
If $d_x \leq \delta_0$, then
\[
\begin{split}
|u(x) - u(y)| & \leq C \Big( \widetilde{\omega}(d_x) + d_x K
\Big)
\left(\frac{|x-y|}{d_x}\right)^{\beta}
\\
& \quad \leq C \Big( \widetilde{\omega}(\delta_0) +  \delta_0 K  \Big)\\
& \qquad \leq \vep
\end{split}
\]
for $\delta_0 > 0$ sufficiently small.
\item If $d_x > \delta_0$, then
(using that $d_x \leq 1$ and $0 < \beta < 1$)
\[
|u(x)-u(y)| 
\leq C \left( \widetilde{\omega}(1) \delta_0^{-\beta} + K \right) \delta^{\beta}
\leq \vep
\]
choosing $\delta > 0$ sufficiently small.
\end{itemize}
\end{proof}

\end{document}